\documentclass[reqno]{amsart}
\pdfoutput=1
\usepackage{amssymb,amsbsy,amsfonts,amssymb,amsmath,mathrsfs}
\usepackage[utf8]{inputenc}
\usepackage[english]{babel}
\usepackage[T1]{fontenc}
\usepackage{times,microtype,xspace}
\usepackage[colorlinks=true,citecolor= DarkBlue,linkcolor=black]{hyperref}
\usepackage[svgnames,table]{xcolor}
\selectcolormodel{gray}
\usepackage[all,cmtip]{xy}
\xyoption{pdf}
\usepackage{graphicx,framed,enumitem,stackrel}
\usepackage{tikz}
\definecolor{smartgreen}{RGB}{191,242,204}
\newcommand{\C}{\mathbb{C}\xspace}  \renewcommand{\P}{\mathbb{P}\xspace} \newcommand{\N}{\mathbb{N}\xspace}  
\renewcommand{\O}{\mathcal{O}\xspace}
\renewcommand{\leq}{\leqslant} \renewcommand{\geq}{\geqslant}
\newcommand{\bydef}{\mathrel{\mathop{:}}=}
\newcommand{\defby}{=\mathrel{\mathop{:}}}
\renewcommand{\bar}[1]{\mkern .5mu\overline{\mkern-.5mu\mathstrut #1\mkern-.5mu}\mkern .5mu}
\DeclareMathOperator{\moinsun}{-1}
\newcommand{\abs}[1]{\lvert#1\rvert} 
\newcommand{\Abs}[1]{\lVert#1\rVert}
\newcommand{\xym}[3][1]{\left.\vcenter{\xy\xymatrix"m"@R=.5pt@C=.5pt@W=1em@H=#1em{#3}\POS"m1,1"."m#2"!C*\frm{(}*\frm{)}\endxy}\right.}
\newcommand{\ad}{\mathbin{\ooalign{%
  \relax\cr
  \hidewidth$\cdot$\hidewidth\cr
  \noalign{\vskip3pt}
  \kern1pt{\tiny\upshape ad}\kern1pt\cr
  \noalign{\vskip-3pt}}%
}}
\newcommand{\vf}[3][z]{\partial\raisebox{-1.3pt}{\scalebox{.8}{$#1_{#2}^{#3}$}}}
\def\mathbi#1{\textbf{\textit #1}}
\newtheorem*{THM}{Main Theorem}
\newtheorem{lemma}{Lemma}
\newtheorem{corollary}{Corollary}
\newtheorem{proposition}{Proposition}
\theoremstyle{definition}
\newtheorem{definition}{Definition}
\usepackage[margin=1.5in]{geometry}
\begin{document}
\title{Slanted Vector Fields for Jet Spaces}
\author{Lionel Darondeau}
\address{Lionel Darondeau\\Laboratoire de Mathématiques d'Orsay\\Université Paris-Sud (France).}
\email{lionel.darondeau@normalesup.org}
\date{26th February, 2015}

\begin{abstract}
  Low pole order frames of slanted vector fields are constructed on the space of vertical \(k\)-jets of the universal family of complete intersections in \(\P^n\) and, adapting the arguments, low pole order frames of slanted vector fields are also constructed on the space of vertical logarithmic \(k\)-jets along the universal family of projective hypersurfaces in \(\P^n\) with several irreducible smooth components.

  Both the pole order (here \(=5k-2\)) and the determination of the locus where the global generation statement fails are improved compared to the literature (previously \(=k^{2}+2k\)), thanks to three new ingredients:
  we reformulate the problem in terms of some adjoint action, we introduce a new formalism of \textsl{geometric jet coordinates}, and then we construct what we call \textsl{building-block vector fields}, making the problem for arbitrary jet order \(k\geq 1\) into a very analog of the much easier case where \(k=0\), \textit{i.e.} where no jet coordinates are needed.

  \medskip

  \noindent\textit{Keywords:}
  {Slanted vector fields, geometric jet coordinates, logarithmic jets, variational method of Voisin-Siu, hyperbolicity, building-block vector fields.}

\end{abstract}
\subjclass{32Q45,14J70,15A03}
\maketitle

\section*{Introduction}
\label{sec:intro}
The formalism of \textsl{jets} is a coordinate-free description of the differential equations that holomorphic curves may satisfy. For a map \(f\colon\C\to X\), valued in a complex projective manifold \(X\), the \(k\)-jet map \(f_{[k]}\colon\C\to J_{k}X\) valued in the \(k\)-jet bundle \(J_{k}X\) corresponds to the truncated Taylor expansion of \(f\) at order \(k\) in some local coordinate system. In \(J_{k}X\), each \textsl{jet-coordinate} \(f_{i}^{(p)}\) shall be considered as an independent coordinate, whence each algebraic differential equation (with holomorphic coefficients) of order \(k\) shall be thought of as a polynomial equation in \(J_{k}X\):
\[
  P(f',f'',\dotsc,f^{(k)})
  \equiv
  0. 
\]
Similarly, if \(D\subset X\) is a normal crossings divisor, the subsheaf \(J_{k}X(-\log D)\subset J_{k}X\) of \textsl{logarithmic \(k\)-jets} on \(X\) along \(D\) can be defined by considering the logarithmic derivatives in the direction of \(D\) (see below).

\subsubsection{Schwarz lemma.}
A (logarithmic) \textsl{\(k\)-jet differential} is locally a polynomial in the (logarithmic) jet-coordinates \(f_{i}^{(p)}\) having constant homogeneous weight, when the weight of \(f_{i}^{(p)}\) is the number of ``primes'' \(p\).
The jet differentials enjoy the following \textsl{fundamental vanishing theorem} (\cite{MR609557,MR1492539,Siu1996}).

\smallskip

\begin{itshape}
If \(\omega\) is a holomorphic jet differential on \(X\) with logarithmic poles along \(D\), that vanishes on an ample divisor,
and if \(f\) is a nonconstant holomorphic map \(\C\to X\setminus D\), then the pullback \(f^{\star}\omega=P(f',\dotsc,f^{(k)})\) vanishes identically on \(\C\).
\end{itshape}

\smallskip

When the canonical divisor \(K_{X}+D\) is big, an interesting question, motivated by the longstanding Green-Griffiths conjecture (\cite{MR609557}), is the algebraic degeneracy of such holomorphic maps \(\C\to X\setminus D\). 
Starting with a lot of differential equations as above, the overall idea is to decrease the degree of the differential equations by elimination of the jet coordinates \(f_{i}^{(p)}\) until obtaining a differential equation of degree \(0\), that is an algebraic equation satisfied by every nonconstant entire curve \(f\colon\C\to X\setminus D\).
We will briefly recall the key points of this strategy, already implemented both in the compact setting (\(X\setminus D=X_{d}\subset\P^{n}\), \cite{MR2593279}) and in the logarithmic setting (\(X\setminus D=\P^{n}\setminus X_{d}\), \cite{arXiv:1402.1396}).
For more details, the reader is referred to the comprehensive recent article \cite{arXiv:1209.2723S} by Yum-Tong Siu. 

\subsubsection{Siu's strategy.}
The general idea is that the vector fields \(\mathsf{V}\in T_{J_{k}X}\) applied to \(\omega\) produce new differential equations. 
However, such equations do not necessarily enjoy the fundamental vanishing theorem, since 
if the pole order of a vector field \(\mathsf{V}\) is bigger than the vanishing order of \(\omega\), then the hypotheses of the theorem are not satisfied by \(\mathsf{V}\cdot\omega\) anymore!
It is thus crucial to control the pole order of these vector fields.

On a regular hypersurface of high degree, there cannot be sufficiently many nonzero meromorphic vector fields having low pole order, but according to a strategy due to Voisin and Siu, in order to get a lot of low pole order vector fields, one can use the positivity of the moduli space\(\abs{\mathcal{O}(d)}\) of all degree \(d\) hypersurfaces in \(\P^{n}\):
\[
  \abs{\mathcal{O}(d)}
  \bydef
  \P\Big(H^{0}\big(\P^{n},\mathcal{O}_{\P^{n}}(d)\big)\Big)
  =
  \P\Big\{\sum_{\abs{\alpha}=d} A_{\alpha}Z^{\alpha}\colon A_{\alpha}\in\C\Big\}
  =
  \P^{n_{d}},
\]
where for \(\alpha\in\N^{n+1}\), \(\abs{\alpha}\bydef \alpha_{0}+\dotsb+\alpha_{n}\) and \(Z^{\alpha}\bydef Z_{0}^{\alpha_{0}}\dotsm Z_{n}^{\alpha_{n}}\) and where \(n_{d}\bydef\binom{n+d}{d}-1\).

In what will be called the \textsl{compact case}, we will consider the universal family 
of complete intersections of multi-degree \(d_{1},\dotsc,d_{c}\):
\[
  \mathcal{Y}_{d_{1},\dotsc,d_{c}}
  \bydef
  \Big\{
    \sum_{\abs{\alpha}=d_{1}} A_{\alpha}^{1}\,Z^{\alpha}
    =
    \dotsb
    =
    \sum_{\abs{\alpha}=d_{c}} A_{\alpha}^{c}\,Z^{\alpha}
    =
    0
  \Big\}
  \subset
  \P^{n}
  \times 
  \abs{\O(d_{1})}\times\dotsb\times\abs{\O(d_{c})},
\] 
(where of course the coefficients \(A_{\alpha}^{j}\) are the homogeneous coordinates on the corresponding \(\abs{\O(d_{j})}\)), together with the space of \(k\)-jets of this universal family \(\mathcal{Y}_{d_{1},\dotsc,d_{c}}\):
\[
  J_{k,d_{1},\dotsc,d_{c}}
  \bydef
  J_{k}\mathcal{Y}_{d_{1},\dotsc,d_{c}}.
\]

In what will be called the \textsl{logarithmic case}, we will consider the universal family 
of normal crossings divisors with \(c\) smooth irreducible components, of respective degrees \(d_{1}\), \dots, \(d_{c}\) in \(\P^{n}\):
\[
  \mathcal{H}_{d_{1},\dotsc,d_{c}}
  \bydef
  \Big\{
    \big(\sum_{\abs{\alpha}=d_{1}} A_{\alpha}^{1}\,Z^{\alpha}\big)
    \dotsm
    \big(\sum_{\abs{\alpha}=d_{c}} A_{\alpha}^{c}\,Z^{\alpha}\big)
    =
    0
  \Big\}
  \subset 
  \P^{n}
  \times 
  \abs{\O(d_{1})}\times\dotsb\times\abs{\O(d_{c})},
\]
together with the space of logarithmic \(k\)-jets along this universal family \(\mathcal{H}_{d_{1},\dotsc,d_{c}}\):
\[
  \bar{J}_{k,d_{1},\dotsc,d_{c}}
  \bydef
  J_{k}(\P^{n}\times\abs{\O(d_{1})}\times\dotsb\times\abs{\O(d_{c})})(-\log\mathcal{H}_{d_{1},\dotsc,d_{c}}).
\]

In both cases, let \(\eta\colon f_{[k]}\mapsto f(0)\) denote the evaluation of the jets and, for shortness, let us denote the space parametrizing the considered universal families by:
\[
  S_{d_{1},\dotsc,d_{c}}
  \bydef
  \abs{\O(d_{1})}\times\dotsb\times\abs{\O(d_{c})}
  =
  \P^{n_{d_{1}}}\times\dotsb\times\P^{n_{d_{c}}}.
\]

The space  of \textsl{vertical \(k\)-jets} is the subspace 
\(J_{k,d_{1},\dotsc,d_{c}}^{\mathit{vert}}\subset J_{k,d_{1},\dotsc,d_{c}}\)
consisting of jets tangent to the fibers of the second projection 
\[
  \mathrm{pr_{2}}
  \colon
  \P^{n}\times S_{d_{1},\dotsc,d_{c}}\to S_{d_{1},\dotsc,d_{c}}.
\]
Similarly, the space  of \textsl{vertical logarithmic \(k\)-jets} is the subspace 
\(\bar{J}_{k,d_{1},\dotsc,d_{c}}^{\mathit{vert}}\subset \bar{J}_{k,d_{1},\dotsc,d_{c}}\)
consisting of logarithmic jets tangent to the fibers of the second projection \(\mathrm{pr_{2}}\).
These jets are introduced in order to use the Schwarz lemma fiberwise (see \cite{MR2593279,arXiv:1402.1396}).

\begin{THM}
  Suppose that the order \(k\) of the jets is smaller than the degrees \(d_{1},\dotsc,d_{c}\), then
  \begin{itemize}
    \setlength\itemsep{1em}
    \item{\normalfont(\textsl{Compact} case)}
      The twisted holomorphic tangent bundle to vertical \(k\)-jets of the universal family of complete intersections of multi-degree \(d_{1},\dotsc,d_{c}\):
      \[
        T_{J_{k,d_{1},\dotsc,d_{c}}^{\mathit{vert}}}
        \otimes
        \eta^{\star}\Bigl(
        \O_{\P^{n}}\bigl(5k-2\bigr)
        \otimes
        \O_{S_{d_{1},\dotsc,d_{c}}}(1,\dotsc,1)
        \Bigr)
      \]
      is generated by its holomorphic global sections at every point of the subspace of \(J_{k,d_{1},\dotsc,d_{c}}^{\mathit{vert}}\) made of \(k\)-jets of non stationary holomorphic curves \(\C\to\mathcal{Y}_{d_{1},\dotsc,d_{c}}\) tangent to the fibers of the second projection
      \(\P^{n}\times S_{d_{1},\dotsc,d_{c}} \to S_{d_{1},\dotsc,d_{c}}\).

    \item{\normalfont(\textsl{Logarithmic} case)}
      The twisted holomorphic tangent bundle to vertical logarithmic \(k\)-jets along the universal family of normal crossing divisors having irreducible smooth components of degrees \(d_{1},\dotsc,d_{c}\)
      \[
        T_{\bar{J}_{k,d_{1},\dotsc,d_{c}}^{\;\mathit{vert}}}
        \otimes
        \eta^{\star}\Bigl(
        \O_{\P^{n}}\bigl(5k-2\bigr)
        \otimes
        \O_{S_{d_{1},\dotsc,d_{c}}}(1,\dotsc,1)
        \Bigr)
      \]
      is generated by its holomorphic global sections at every point of the subspace of \(\bar{J}_{k,d_{1},\dotsc,d_{c}}^{\;\mathit{vert}}\) made of logarithmic \(k\)-jets of non stationary holomorphic curves 
      \[
      \C\to(\P^{n}\times S_{d_{1},\dotsc,d_{c}})\setminus\mathcal{H}_{d_{1},\dotsc,d_{c}}
      \]
      tangent to the fibers of the second projection
      \(\P^{n}\times S_{d_{1},\dotsc,d_{c}} \to S_{d_{1},\dotsc,d_{c}}\).
  \end{itemize}
\end{THM}

Actually, the extrinsic vector fields that generate the tangent bundle are not tangential to the space of \(k\)-jets of the vertical fiber at a generic point, what justifies to term them \textsl{slanted vector field} as Siu does.

When low pole order meromorphic slanted vector fields are used in order to eliminate the derivatives \(f',\dotsc,f^{(k)}\), following Siu's strategy, our main Theorem essentially yields that actually every single algebraic coefficient (depending only on \(f\)) of each algebraic differential equation \(P(f',\dotsc,f^{(k)})\equiv0\) has to vanish.
For a more detailed account, see \cite{arXiv:1209.2723S,MR2593279,arXiv:1402.1396},
as we will now focus on the proof of our statement.

\subsubsection{State of the art.}
The method of slanted vector fields has been introduced by Siu (\cite{MR2077584}), and is motivated by the work of Clemens-Ein-Voisin (\cite{MR875091,MR958594,MR1420353}) on rational curves.
It has been pushed further in dimension \(2\) by P\u{a}un (\cite{MR2372741}) for the compact case and by Rousseau (\cite{MR2552951}) for the logarithmic case, with several smooth components, and in dimension \(3\) by Rousseau, both for the compact case (\cite{MR2331545}) and for the logarithmic case (\cite{MR2383820}).
In the compact case, the technique has been generalized in any dimension by Merker (\cite{MR2543663}), with a substantial improvement of the determination of the locus where the global generation statement fails, leading to a proof of the \emph{strong} algebraic degeneracy of entire curves with values in a generic projective hypersurface of effective large degree (\cite{MR2593279}).
In the slightly different context of projective hypersurfaces in families, Mourougane (\cite{MR2911888}) has implemented the technique in any dimension and for any jet order.

\subsubsection{Organization of the paper.}
In \S1, we recall the formalism of jets and we introduce new \textsl{geometric} jet coordinates, together with the associated vector fields, that exist on the subspace of invertible jets. 
When all is said and done, it will appear that this formalism significantly simplifies the computations, because it allows to take advantage of the triangular properties of the iterations of the chain rule, which rule of course plays a central role when dealing with jet spaces.

In \S2, we describe a new general strategy to easily construct slanted vector fields on the subspace of invertible jets, following the formal differentiation presentation of Merker in \cite{MR2543663}, and we implement this strategy in the compact case, for \(c=1\).
In the directions tangent to the space spanned by the jet coordinates, starting with guiding examples,  we retrieve the results presented in the work of P\u{a}un (\cite{MR2372741}), Rousseau (\cite{MR2331545,MR2383820,MR2552951}) and Merker (\cite{MR2543663}) by considering the first order geometric jet coordinates, and then we go further by using the higher order jet coordinates. This allows to clarify the locus where the global generation fails, a detail that was only briefly treated in the previous works.
In the directions tangent to the space of parameters, the arguments introduced by Merker in \cite{MR2543663} could directly apply, and we merely reformulate them in terms of geometric jet coordinates, which will prove to be useful for computing more accurately the pole order of the meromorphic prolongations of the described vector fields.

In \S3, we prove the main theorem on global generation of the tangent bundle to \(k\)-jet spaces.
In the case of the universal hypersurface, considered in the previous section, we compute the pole order of the frame of slanted vector fields that we have constructed. Then we adapt the result to the case of general complete intersections, with \(c\geq1\). After that, we also adapt the result to the logarithmic case.
For this, we follow the strategy implemented by Rousseau in \cite{MR2552951,MR2383820} for complements of hypersurfaces in \(\P^{2}\) and \(\P^3\), by locally straightening out the universal family \(\mathcal{H}_{d_{1},\dotsc,d_{c}}\).

\section{Geometric Coordinates for Invertible Jets.}
\label{sec:logJetBdl}
So we will define the \textsl{geometric jet coordinates}, that will play a pivotal role in our simplifying strategy, and then study the associated vector fields.
Let us start with some standard material.

\subsubsection{Logarithmic jet bundles.}
\label{sse:jetMfld}
Given a point \(x\) in a \(n\)-dimensional complex manifold \(X\),
the \textsl{\(k\)-jet} of a germ of holomorphic map \(f\colon(\C,0)\to(X,x)\) is the equivalence class \(f_{[k]}\) of germs that osculate with \(f\) to order \(k\) at the origin of \(\C\). 
Fixing some local coordinate system \((z_{1},\dotsc,z_{n})\) over an open subset \(U\subset X\) around \(x\),
the space of \(k\)-jets \(J_{k}X\vert_{x}\) therefore identifies, by Taylor formula, to the vector space \(\C^{nk}\) using the map:
\begin{equation}
  \label{eq:trivialisation}
  f_{[k]}
  \mapsto
  \big(
  f_{1}',\dotsc, f_{n}',
  \dotsc,
  f_{1}^{(k)},\dotsc,f_{n}^{(k)}
  \big).
\end{equation}
Here and throughout the rest of this text, we will rather use the notation \(f^{(j)}\) for the \(j\)-th Taylor coefficient than for the \(j\)-th derivative, \textit{i.e.}:
\[
  f^{(j)}
  \bydef
  \frac{1}{j!}
  \frac{d^{j}f}{dt^{j}}.
\]
The collection of these vector spaces gives rise to a holomorphic fiber bundle \(J_{k}X\), which is called the \textsl{\(k\)-jet manifold} of \(X\).

To get a local trivialization of \(J_{k}X\) around a point \(x\in X\), one can of course use the map \eqref{eq:trivialisation}. A construction due to Noguchi~(\cite{MR859200}) allows more general ``derivatives'', that are potentially more adapted to the geometric situation that will be dealt with below. Here ``derivative'' means pullback of local meromorphic \(1\)-forms. 

Let thus \(\omega\in T_{U}^{\star}\) be a local meromorphic \(1\)-form over an open subset \(U\subset X\) and let \(f\colon\Omega\to U\) be a local holomorphic map over an open subset \(\Omega\subset\C\). This map \(f\) induces a meromorphic function \(A'\colon\Omega\to\C\) by the formula:
\[
  f^{\star}\omega\vert_{t}
  \defby
  A'(t)\,dt,
\]
where we equip the complex plane \(\C\) with the standard complex coordinate \(t\). The obtained map \(A'\) depends only on the \(1\)-jet of \(f\). More generally, the \(j\)-th derivative of \(A'\) (up to \(j=k-1\)) is well defined and depends only on the \((j+1)\)-jet of \(f\).
One gets a meromorphic map \(J_{k}X\vert_{U}\to\C^{k}\) associated to \(\omega\) by defining:
\begin{equation}
  \label{eq:noguchi}
  \widetilde{\omega}
  \colon
  f_{[k]}
  \mapsto
  \left(A',\frac{1}{2!}\frac{dA'}{dt},\dotsc,\frac{1}{k!}\frac{d^{k-1}A'}{dt^{k-1}}\right).
\end{equation}
Note that if \(\omega\) is taken holomorphic, then \(\widetilde{\omega}\) becomes also holomorphic.

In the simplest case where \(\omega=dz_{i}\) is locally exact, one gets the derivatives in the direction \(z_{i}\). In the basic particular case where \(\omega=dz_{i}/z_{i}\), one gets the logarithmic derivatives in the direction \(z_{i}\).

For any holomorphic frame \((\omega^{1},\dotsc,\omega^{n})\) of the holomorphic cotangent bundle \(T_{X}^{\star}\), the map \((\widetilde{\omega}^{1},\dotsc,\widetilde{\omega}^{n})\) is a trivialization of the fiber of \(J_{k}X\).

\medskip

Recall that a \textsl{normal crossing divisor} \(D\) is a reduced divisor that looks locally like a (possibly empty) union of coordinate hyperplanes: for any point \(x\in X\), there exists a neighbourhood \(U\) at \(x\) with a local holomorphic coordinate system \((z_{1},\dotsc,z_{n})\) such that \(U\cap D\) is the zero locus \(\{z_{1}\dotsm z_{\ell}=0\}\), for some integer \(\ell\leq n\) that depends on \(x\).
The \textsl{logarithmic cotangent sheaf} \(T_{X}^{\star}(\log D)\) is then defined as the locally free subsheaf of the sheaf of meromorphic \(1\)-forms on \(X\) whose stalk at any point \(x\) is defined by:
\[
  T_{X}^{\star}(\log D)\vert_{x}
  \bydef
  \sum_{i=1}^{\ell}
  \mathcal{O}_{X,x}\;\frac{dz_{i}}{z_{i}}
  +
  \sum_{i=\ell+1}^{n}
  \mathcal{O}_{X,x}\;{dz_{i}},
\]
for any \textsl{logarithmic coordinate system} \((z_{1},\dotsc,z_{n})\) along \(D\) at \(x\), such that \(U\cap D=\{z_{1}\dotsm z_{\ell}=0\}\).

Recall after \cite{MR859200} that a holomorphic section \(s\in H^{0}(U,J_{k}X)\) on an open subset \(U\subset X\) is said to be a \textsl{logarithmic jet field} along \(D\) of order \(k\) if
for all logarithmic \(1\)-forms \(\omega\) defined on an arbitrary open subset \(V\subset U\)
the map:
\[
  \widetilde{\omega}
  \circ
  s\vert_{V}
  \colon
  V
  \to
  \C^{k}
\]
is holomorphic. These sections define a subsheaf of the sheaf of sections of \(J_{k}X\) that is itself the sheaf of sections of a holomorphic affine bundle \(J_{k}(X,-\log D)\), called the \textsl{logarithmic jet bundle of order \(k\) along \(D\)}. 
For any holomorphic frame \((\omega^{1},\dotsc,\omega^{n})\) of the logarithmic cotangent bundle \(T_{X}^{\star}(\log D)\), the map \((\widetilde{\omega}^{1},\dotsc,\widetilde{\omega}^{n})\) is a trivialization of the fiber of \(J_{k}(X,-\log D)\). 

A detailed description of the properties of the bundle \(J_{k}(X,-\log D)\) can be found in the paper \cite{MR1824906} by Dethloff and Lu.

\subsubsection{Geometric jet coordinates (1).}
\label{sse:geometricJetCoordinates}
A jet field \(j\in J_{k}X\) is termed \textsl{singular} at a point \(x\in X\) if it is the lift of a stationary curve, \textit{i.e.} \(j=f_{[k]}\) with \(f_{1}'(x)=\dotsb=f_{n}'(x)=0\). The subset of singular jets will be denoted by \(J_{k}^{\mathit{sing}}X\).
Note that for a logarithmic pair \((X,D)\), the logarithmic jet bundle \(J_{k}(X,-\log D)\vert_{X\setminus D}\) and the holomorphic jet bundle \(J_{k}(X\setminus D)\) coincide on the open part \(X\setminus D\).
This observation allows to define singular logarithmic jet fields, as follows.
A logarithmic jet field is said \textsl{singular} if it is in the topological closure of the subset \(J_{k}^{\mathit{sing}}(X\setminus D)\) in \(J_{k}(X,-\log D)\).
A (logarithmic) jet field that is not singular is termed \textsl{invertible} (or \textsl{regular}).

For a short moment, we will work in the so-called \textsl{compact case}, where \(D\) is empty, because the modifications needed to treat the general case are straightforward. We fix a coordinate system \((z_{1},\dotsc,z_{n})\) on an open set of \(X\).
Let:
\[
  z_{i}^{(p)}\colon f_{[k]}\mapsto f_{i}^{(p)}
\]
denote the jet coordinates obtained for the holomorphic frame \(dz_{1},\dotsc,dz_{n}\) according to \eqref{eq:noguchi}.
Consider the invertible jets, for which at least one first derivative \(z_{i_{0}}'\) is not zero. 

Without loss of generality, assume that \(i_{0}=1\).
By the local inverse theorem, the map \(f_{1}\colon\C\to\C\) is locally invertible, and it is very natural to use it as the complex variable. 
To do so, let us first describe how jet variables behave under reparametrization at the source.

Let \(g\colon\C\to X\) be a (local) map of class \(\mathscr{C}^{k}\) and let \(h\colon\C\to\C\) be a (local) reparametrization of the source of class \(\mathscr{C}^{k}\). 
Recall that the \textsl{length} of a multi-index \(\mu\in\N^{k}\) is by definition the sum:
\[
  \abs{\mu}\bydef \mu_{1}+\mu_{2}+\dotsb+\mu_{k},
\]
and define its \textsl{weight} to be the weighted sum:
\[
  \Abs{\mu}\bydef \mu_{1}+2\mu_{2}+\dotsb+k\mu_{k}.
\]
Classically, for any integer \(p=0,1,\dotsc,k\), the \(p\)-th Taylor coefficient
\(g^{(p)}=\frac{1}{p!}\frac{d^{p}g}{dt^{p}}\) of the map \(g\) is 
related to those of the reparametrized map \(g\circ h^{\moinsun}\)
by the \textsl{Faà di Bruno formula}~(\cite[§3.4, p.137]{MR0460128}):
\begin{equation}
  \label{eq:faa1}
  g^{(p)}
  =
  \sum_{q\leq p}
  \mathbi{B}_{p,q}\big(h^{(1)},\dotsc,h^{(k)}\big)\,
  \left(g\circ h^{\moinsun}\right)^{(q)}
  \circ h,
\end{equation}
where the \(\mathbi{B}_{p,q}\big(h^{(1)},\dotsc,h^{(k)}\big)\), hereafter denoted by \(\mathbi{B}_{p,q}(h)\) for shortness, are the so-called \textsl{Bell polynomials}~(\cite[§3.3, p.133]{MR0460128}):
\begin{equation*}
  \label{eq:bell}
  \mathbi{B}_{p,q}(h)
  \bydef
  \sum_{\Abs{\mu}=p,\abs{\mu}=q}
  \frac{\abs{\mu}!}{\mu!}\,
  {h^{(1)}}^{\mu_{1}}\dotsm{h^{(k)}}^{\mu_{k}},
\end{equation*}
with
\(
\mu!
\bydef
\mu_{1}!\dotsm\mu_{k}!
\).

Since the weighted sum
\(
\Abs{\mu}
=
1\,\mu_{1}+
2\,\mu_{2}+
\dotsb+
k\,\mu_{k}
\) 
is clearly never less than the plain sum
\(
\abs{\mu}
=
1\,\mu_{1}+
1\,\mu_{2}+
\dotsb+
1\,\mu_{k},
\)
the equality case \(\Abs{\mu}=\abs{\mu}=p\) implying \(\mu=(p,0\dotsc,0)\),
it is immediate that \(\mathbi{B}_{p,q}=0\) for \(p<q\) and \(\mathbi{B}_{p,p}(h)=(h^{(1)})^{p}=(h')^{p}\) for \(p=q\). 

Adopting the following notation for the Taylor coefficients in the summand of \eqref{eq:faa1}:
\begin{equation*}
  \frac{1}{q!}
  \frac{d^{q} g}{dh^{q}}
  \bydef
  \left(g\circ h^{\moinsun}\right)^{(q)}
  \circ h,
\end{equation*}
we gather the formulas \eqref{eq:faa1} for \(p=1,\dotsc,k\) 
under the form of the following invertible matricial equation:
\begin{equation}
  \label{eq:faa.mat}
  \xym[.57]{4,1}{
    g^{(1)}\ar@{.}[ddd]\\
    \\
    \\
    g^{(k)}
  }
  =
  \xym[0]{5,4}
  {
    h'\ar@{.}[4,3]&0\ar@{.}[0,2]&&0\\
    &&&\\
    &&&\\
    &&&0\ar@{.}@{.}[-3,-2]\ar@{.}@{.}[-3,0]\\
    &\mathbi{B}_{p,q}(h)&&(h')^{k}\\
  }
  \xym[1.08]{3,1}{
    \frac{1}{1!}\frac{dg}{dh}\ar@{.}[dd]\\
    \\
    \frac{1}{k!}\frac{d^{k}g}{dh^{k}}
  }.
\end{equation}

\medskip

We now come back to the particular situation of the set \(\{z_{1}'\neq0\}\) and we introduce new jet coordinates as follows.
\begin{definition}
  For a \(k\)-jet of holomorphic curve \(f\colon(\C,0)\to(X,x)\) with \(f_{1}'(0)\neq0\), and for \(i\neq1\), the \textsl{geometric jet coordinates} of \(f\) in the \(z_{i}\)-direction at \(x\) are:
  \[
    \bigg(f_{i}^{[1]},f_{i}^{[2]},\dotsc,f_{i}^{[k]}\bigg)
    \bydef
    \left(
    \frac{df_{i}}{df_{1}},
    \frac{1}{2!}\frac{d^{2}f_{i}}{d{f_{1}}^{2}},
    \dotsc,
    \frac{1}{k!}\frac{d^{k}f_{i}}{d{f_{1}}^{k}}
    \right).
  \]
\end{definition}

Notice that the \(f_{i}^{[p]}\) are genuine Taylor coefficients, but we use square brackets so that the introduced geometric jet coordinates cannot be confused with the usual jet coordinates \(f_{i}^{(p)}\).

The geometric jet coordinates on \(\{z_{1}'\neq0\}\) in the \(z_{i}\)-direction are defined to be \(z_{i}^{[p]}\colon f_{[k]}\mapsto f_{i}^{[p]}\).
By lifting \eqref{eq:faa.mat}, with \(h=z_{1}\), at the level of \(k\)-jets, one obtains the following relations between standard jet coordinates and geometric jet coordinates.
\begin{equation}
  \label{eq:geom.jet.i}
  z_{i}^{(p)}
  =
  \sum_{p\geq q}
  \mathbi{B}_{p,q}(z_{1})\;
  z_{i}^{[q]}
  \qquad
  {\scriptstyle(i=2,\dotsc,n)}.
\end{equation}

In the \(z_{1}\)-direction, the same definition would produce ``\(z_{1}^{[1]}=1\)'' and then ``\(z_{1}^{[p]}=0\)'' for \(p\geq2\).
Thus in order to complete the jet coordinate system we rather use :
\[
  t^{[p]}\colon f_{[k]}\mapsto
  \frac{d^{p}t}{df_{1}^{p}}
  \bydef
  ({f_{1}}^{\moinsun})^{(p)}\circ f_{1},
\]
where of course \(t\) stands for the identity map \(t\mapsto t\). Stating that \(t^{(q)}=\delta_{1,q}\), one infers:
\[
  \delta_{1,p}
  =
  \sum_{p\geq q}
  \mathbi{B}_{p,q}(z_{1})\;
  t^{[q]}
  \qquad
  {\scriptstyle(i=2,\dotsc,n)}.
\]

To sum up, the two systems of coordinates are related by:
\begin{equation}
  \label{eq:geometricCoordinates}
  \xym{4,5}{
    1&z_{1}^{(1)}&z_{2}^{(1)}\ar@{.}[rr]&&z_{n}^{(1)}\\
    0\ar@{.}[dd]&z_{1}^{(2)}\ar@{.}[dd]&z_{2}^{(2)}\ar@{.}[rr]\ar@{.}[dd]&&z_{n}^{(2)}\ar@{.}[dd]\\
    &&&&\\
    0&z_{1}^{(k)}&z_{2}^{(k)}\ar@{.}[rr]&&z_{n}^{(k)}
  }
  =
  \mathbi{B}(z_{1})\;
  \xym{4,5}{
    t^{[1]}&1&z_{2}^{[1]}\ar@{.}[rr]&&z_{n}^{[1]}\\
    t^{[2]}\ar@{.}[dd]&0\ar@{.}[dd]&z_{2}^{[2]}\ar@{.}[rr]\ar@{.}[dd]&&z_{n}^{[2]}\ar@{.}[dd]\\
    &&&&\\
    t^{[k]}&0&z_{2}^{[k]}\ar@{.}[rr]&&z_{n}^{[k]}
  }.
\end{equation}

Reciprocally, it is a well-known fact that Bell arrays behave well with respect to composition; indeed  (\cite[§3.7, p145]{MR0460128}):
\[
\mathbi{B}(g_{1}\circ g_{2})
=
\mathbi{B}(g_{2})\;
\big(\mathbi{B}(g_{1})\circ g_{2}\big).
\]
By taking \(g_{1}=z_{1}\) and \(g_{2}={z_{1}}^{\moinsun}\), one obtains
\(
\mathbi{I}_{k}
=
\mathbi{B}(z_{1})\;
\big(\mathbi{B}({z_{1}}^{\moinsun})\circ z_{1}\big)
\),
so we define:
\begin{equation}
  \label{eq:Bt}
  \mathbi{B}_{p,q}[t]
  \bydef
  \mathbi{B}_{p,q}({z_{1}}^{\moinsun})\circ z_{1}
  =
  \sum_{\Abs{\mu}=p,\ \abs{\mu}=q}
  \frac{\abs{\mu}!}{\mu!}\,
  {t^{[1]}}^{\mu_{1}}
  \dotsm
  {t^{[k]}}^{\mu_{k}},
\end{equation}
and the two matrices \(\mathbi{B}(z_{1})\) and \(\mathbi{B}[t]\) are inverse of each other. 
Thus:
\begin{equation}
  \xym{4,5}{
    t^{[1]}&1&z_{2}^{[1]}\ar@{.}[rr]&&z_{n}^{[1]}\\
    t^{[2]}\ar@{.}[dd]&0\ar@{.}[dd]&z_{2}^{[2]}\ar@{.}[rr]\ar@{.}[dd]&&z_{n}^{[2]}\ar@{.}[dd]\\
    &&&&\\
    t^{[k]}&0&z_{2}^{[k]}\ar@{.}[rr]&&z_{n}^{[k]}
  }
  =
  \mathbi{B}[t]\;
  \xym{4,5}{
    1&z_{1}^{(1)}&z_{2}^{(1)}\ar@{.}[rr]&&z_{n}^{(1)}\\
    0\ar@{.}[dd]&z_{1}^{(2)}\ar@{.}[dd]&z_{2}^{(2)}\ar@{.}[rr]\ar@{.}[dd]&&z_{n}^{(2)}\ar@{.}[dd]\\
    &&&&\\
    0&z_{1}^{(k)}&z_{2}^{(k)}\ar@{.}[rr]&&z_{n}^{(k)}
  }.
\end{equation}

\medskip

\begin{slshape}
  To adapt the definitions to the logarithmic case, we can do all the same reasoning, provided
  \((z_{1},\dotsc,z_{n})\) is a logarithmic coordinate system along \(D=(z_{1}\dotsm z_{\ell}=0)\),
  by taking the jet-coordinate associated to the frame \(({dz_{1}}/{z_{1}},\dotsc,{dz_{\ell}}/{z_{\ell}},dz_{\ell+1},\dotsc,dz_{n})\), namely for \(p=1,\dotsc,k\): 
  \[
    \begin{cases}
      z_{i}^{(p)}\colon f_{[k]}\mapsto (\log f_{i})^{(p)}&\text{for }i=1,\dotsc,\ell,\\
      z_{i}^{(p)}\colon f_{[k]}\mapsto f_{i}^{(p)}&\text{for }i=\ell+1,\dotsc,n.
    \end{cases}
  \]
\end{slshape}

\subsubsection{Associated vector fields.}
Since \(\mathbi{B}(z_{1})=\mathbi{B}[t]^{\moinsun}\), the coefficients of \(\mathbi{B}(z_{1})\)
have expressions that depend only on \(t^{[1]},\dotsc,t^{[k]}\).
Consequently, \(z_{1}^{(1)},\dotsc,z_{1}^{(k)}\) depend only on \(t^{[1]},\dotsc,t^{[k]}\) because they are the coefficients of the first column of \(\mathbi{B}(z_{1})\), and for \(i=2,\dotsc,n\), the system \eqref{eq:geom.jet.i}:
\[
  z_{i}^{(p)}
  =
  \sum_{p\geq q}
  \mathbi{B}_{p,q}(z_{1})\;
  z_{i}^{[q]}
\]
yields that \(z_{i}^{(p)}\) depends only on \(t^{[1]},\dotsc,t^{[k]}\)
and \(z_{i}^{[1]},\dotsc,z_{i}^{[q]}\) for the single same index \(i\).

One infers the dual relations for the associated vector fields \(\vf{i}{[p]}\) and \(\vf{i}{(q)}\), by plain transposition:
\begin{equation}
  \vf{i}{[p]}
  =
  \sum_{q=p}^{k}
  \mathbi{B}_{q,p}(z_{1})\,
  \vf{i}{(q)},
\end{equation}
where
\(
\vf[\bullet]{}{}
=
\frac{\partial}{\partial\bullet}
\). 
The simplest instances of such vector fields are constructed for \(i=2,\dotsc,n\) in the case \(p=1\), where these vector fields
\begin{equation}
  \label{eq:vfpaun}
  \vf{i}{[1]}
  =
  \sum_{q=1}^{k}
  z_{1}^{(q)}
  \vf{i}{(q)}
\end{equation}
are the same as the vector field implicitly used in the matrix approach presented in the work of P\u{a}un (\cite{MR2372741}).

Because the Bell array \(\mathbi{B}(z_{1})\) is lower triangular with invertible diagonal entries \((z_{1}')^{q}\), one has immediately the following:
\begin{corollary}
  \label{cor:span.i}
  The vector fields \(\vf{i}{[q]}\) span the tangent vector space in the \(z_{i}\)-direction at points where \(z_{1}'\neq0\).
\end{corollary}

\smallskip

In the ``\(t\)-direction'' the vector fields have of course more involved expressions, as every jet coordinate \(z_{i}^{(p)}\) depends on the coefficients \(t^{[q]}\).
Nevertheless, some special linear combination of \(\vf[t]{}{[1]},\dotsc,\vf[t]{}{[k]}\) have simple expressions. 
A first example is the \textsl{Euler vector field}:
\[
  \mathsf{T}_{1}
  \bydef
  t^{[1]}\vf[t]{}{[1]}
  +
  \dotsb
  +
  t^{[k]}\vf[t]{}{[k]}
  =
  -
  \sum_{i=1,\dotsc,n}\;
  \sum_{p=1}^{k}
  p\;z_{i}^{(p)}\,\vf{i}{(p)}.
\]
More generally, we claim that:
\begin {equation}
  \label{eq:tgt.sym}
  \mathsf{T}_{\ell}
  \bydef
  \sum_{m=1}^{k}
  \mathbi{B}_{m,\ell}[t]\,
  \vf[t]{}{[m]}
  =
  -
  \sum_{i=1,\dotsc,n}\;
  \sum_{p=1}^{k-\ell+1}
  p\,z_{i}^{(p)}\,\vf{i}{(p+\ell-1)}.
\end {equation}
\begin{proof}
  Assume that \(\mathsf{V}\) is in the vector space spanned by \(\vf[t]{}{[1]},\dotsc,\vf[t]{}{[1]}\) then:
  \[
    \mathsf{V}
    \cdot
    \textrm{Mat}\big(z_{i}^{(p)}\big)
    =
    \mathsf{V}
    \cdot
    \left(
    \mathbi{B}(z_{1})\;
    \textrm{Mat}\big(z_{i}^{[p]}\big)
    \right)
    =
    \big(
    \mathsf{V}
    \cdot
    \mathbi{B}(z_{1})
    \big)\;
    \textrm{Mat}\big(z_{i}^{[p]}\big),
  \]
  where for shortness:
  \[
    \textrm{Mat}\left(z_{i}^{(p)}\right)
    \bydef
    \xym[0]{4,4}{
      z_{1}^{(1)}&z_{2}^{(1)}\ar@{.}[rr]&&z_{n}^{(1)}\\
      z_{1}^{(2)}\ar@{.}[dd]&z_{2}^{(2)}\ar@{.}[rr]\ar@{.}[dd]&&z_{n}^{(2)}\ar@{.}[dd]\\
      &&&\\
      z_{1}^{(k)}&z_{2}^{(k)}\ar@{.}[rr]&&z_{n}^{(k)}
    }
    \quad
    \text{and}
    \quad
    \textrm{Mat}\left(z_{i}^{[p]}\right)
    \bydef
    \xym[0]{4,4}{
      \rule{0pt}{11pt}1&z_{2}^{[1]}\ar@{.}[rr]&&z_{n}^{[1]}\\
      0\ar@{.}[dd]&z_{2}^{[2]}\ar@{.}[rr]\ar@{.}[dd]&&z_{n}^{[2]}\ar@{.}[dd]\\
      &&&\\
      0&z_{2}^{[k]}\ar@{.}[rr]&&z_{n}^{[k]}
    };
  \]
  the last equality holds by Leibniz rule, because the coefficient of \(\mathrm{Mat}\big(z_{i}^{[p]}\big)\) are independent of \(t^{[1]},\dotsc,t^{[k]}\).
  Now, because \(\mathbi{B}[t]\) and \(\mathbi{B}(z_{1})\) are inverse matrices, one has:
  \[
    \mathsf{V}
    \cdot
    \mathbi{B}(z_{1})
    =
    -
    \mathbi{B}[t]^{\moinsun}
    \big(
    \mathsf{V}
    \cdot
    \mathbi{B}[t]
    \big)
    \mathbi{B}(z_{1}),
  \]
  whence:
  \[
    \mathsf{V}
    \cdot
    \textrm{Mat}\big(z_{i}^{(p)}\big)
    =
    -
    \mathbi{B}[t]^{\moinsun}
    \big(
    \mathsf{V}
    \cdot
    \mathbi{B}[t]
    \big)
    \mathbi{B}(z_{1})\;
    \textrm{Mat}\big(z_{i}^{[p]}\big)
    =
    -
    \mathbi{B}[t]^{\moinsun}
    \big(
    \mathsf{V}
    \cdot
    \mathbi{B}[t]
    \big)\;
    \textrm{Mat}\big(z_{i}^{(p)}\big).
  \]

  It is in general rather difficult to compute 
  \(
  \mathsf{V}
  \cdot 
  \mathbi{B}[t]
  \),
  but in the particular cases that we consider using the very definition \eqref{eq:Bt}, one proves that:
  \[
    \sum_{m=1}^{k}
    \mathbi{B}_{m,\ell}[t]\,
    \vf[t]{}{[m]}
    \cdot
    \mathbi{B}_{p,q}[t]
    =
    \ell\;
    \mathbi{B}_{p,q+\ell-1}[t].
  \]
  which is equivalent to:
  \[
    \sum_{m=1}^{k}
    \mathbi{B}_{m,\ell}[t]\,
    \vf[t]{}{[m]}
    \cdot
    \mathbi{B}[t]
    =
    \mathbi{B}[t]
    \big(
    1\,\mathbi{e}_{\ell}^{1}+
    \dotsb+
    (1+k-\ell)\,\mathbi{e}_{1+k-\ell}^{k}
    \big),
  \]
  where the \(\mathbi{e}_{i}^{j}\) are the elements of the canonical basis of \(\mathrm{Mat}_{k,k}(\C)\).
  Thus:
  \[
    \sum_{m=1}^{k}
    \mathbi{B}_{m,\ell}[t]\,
    \vf[t]{}{[m]}
    \cdot
    \textrm{Mat}\big(z_{i}^{(p)}\big)
    =
    -
    \big(
    1\,\mathbi{e}_{\ell}^{1}+
    \dotsb+
    (1+k-\ell)\,\mathbi{e}_{1+k-\ell}^{k}
    \big)\;
    \textrm{Mat}\big(z_{i}^{(p)}\big),
  \]
  which finally yields the announced result:
  \[
    \sum_{m=1}^{k}
    \mathbi{B}_{m,\ell}[t]\,
    \vf[t]{}{[m]}
    =
    \sum_{i=1,\dotsc,n}\;
    \sum_{p=1}^{k}
    \sum_{m=1}^{k}
    \mathbi{B}_{m,\ell}[t]\,
    \vf[t]{}{[m]}
    \cdot
    z_{i}^{(p)}\vf{i}{(p)}
    =
    -
    \sum_{i=1,\dotsc,n}\;
    \sum_{p=1}^{k-\ell+1}
    p\,z_{i}^{(p)}\,\vf{i}{(p+\ell-1)}.
    \qedhere
  \]
\end{proof}

\begin{corollary}
  \label{cor:span.z}
  The vector fields \(\mathsf{T}_{1},\dotsc,\mathsf{T}_{k}\) and \(\vf{2}{[1]},\dotsc,\vf{2}{[k]}\), \dots, \(\vf{n}{[1]},\dotsc,\vf{n}{[k]}\) span the tangent vector space to \(J_{k}X\) at points where \(z_{1}'\neq0\).
\end{corollary}
\begin{proof}
  By corollary \ref{cor:span.i}, it suffices to prove that the vector fields \(\vf{1}{(1)},\dotsc,\vf{1}{(q)},\dotsc,\vf{1}{(k)}\) are spanned. This is an easy descending induction on \(q=k,\dotsc,1\), since by \eqref{eq:tgt.sym}:
  \[
    \vf{i}{(q)}
    =
    \frac{1}{q\,z_{1}'}
    \bigg(
    \mathsf{T}_{q}
    -
    \sum_{p=q+1}^{k}
    p\;z_{1}^{(1+p-q)}\,\vf{1}{(p)}
    -
    \sum_{i=2,\dotsc,n}
    \sum_{p=q}^{k}
    p\;z_{i}^{(1+p-q)}\,\vf{i}{(p)}
    \bigg).
    \qedhere
  \]
\end{proof}

\subsubsection{Geometric jet coordinates (2).}
To be fully efficient, we will also need an alternative description of the geometric jet coordinates, using differential operators, in the spirit of the presentation of Merker in \cite{MR2543663}.
The \textsl{formal differentiation of jets} \(\mathsf{D}_{t}\) is by definition the following vector field on \(J_{k}X\), thought of as a differential operator on \(k\)-jets, that mimics the differentiation with respect to the standard coordinate \(t\in\C\):
\begin{equation}
  \label{eq:formalDifferentiation}
  \mathsf{D}_{t}
  \bydef
  \sum_{i=1,\dotsc,n}
  \left(
  \sum_{p=0}^{k-1}
  (p+1)\;
  z_{i}^{(p+1)}\;
  \vf{i}{(p)}
  \right).
\end{equation}
Here, the coefficient \((p+1)\) appears only because we use the Taylor coefficients instead of the derivatives.
On the set \(\{z_{1}'\neq0\}\) it is natural to introduce the linear differential operator \(\mathsf{D}_{z_{1}}\), that mimics the total derivative with respect to \(z_{1}\).
These two linear differential operators should be related by the usual chain rule on \(\{z_{1}'\neq0\}\), that we take as a definition:
\begin{equation}
  \label{eq:chainrule}
  \mathsf{D}_{z_{1}}
  \bydef
  \frac{1}{z_{1}'}
  \mathsf{D}_{t}.
\end{equation}
More generally, we claim that this chain rule implies that
on  the set \(\{z_{1}'\neq0\}\), the powers of the differential operators \(\mathsf{D}_{t}\) and \(\mathsf{D}_{z_{1}}\) are related by the invertible triangular system:
\begin{equation}
  \label{cor:DtD1}
  \frac{\mathsf{D}_{t}^{p}}{p!}
  =
  \sum_{q=1}^{p}
  \mathbi{B}_{p,q}(z_{1})\;
  \frac{\mathsf{D}_{z_{1}}^{q}}{q!}
  \qquad
  {\scriptstyle(p=1,\dotsc,k)}.
\end{equation}
This is a purely combinatorial fact, that we will admit, because the proof is not very interesting for our problem.

Throughout what follows, by convention, \(\mathsf{D}_{t}(t)\bydef1\).
\begin{lemma}
  \label{lem:geometric.jets.Dz}
  One has as expected \(\mathsf{D}_{z_{1}}(z_{1})=1\) and for \(p=0,1,\dotsc,k\) one has also:
  \[
    t^{[p]}
    =
    \frac{1}{p!}\mathsf{D}_{z_{1}}^{p}(t),
    \qquad
    z_{i}^{[p]}
    =
    \frac{1}{p!}
    \mathsf{D}_{z_{1}}^{p}(z_{i})
    \quad{\scriptstyle(i=2,\dotsc,n)},
  \]
  where according to the above convention \(\mathsf{D}_{z_{1}}^{p}(t)=\mathsf{D}_{z_{1}}^{p-1}(1/z_{1}')\).
\end{lemma}
\begin{proof}
  The result follows from the matricial equalities:
    \[
      \vcenter{
        \xy
        \xymatrix"Ma"@R=.5pt@C=.5pt@W=1em@H=1em{
          1&z_{1}^{(1)}&z_{2}^{(1)}\ar@{.}[rr]&&z_{n}^{(1)}\\
          0\ar@{.}[dd]&z_{1}^{(2)}\ar@{.}[dd]&z_{2}^{(2)}\ar@{.}[rr]\ar@{.}[dd]&&z_{n}^{(2)}\ar@{.}[dd]\\
          &&&&\\
          0&z_{1}^{(k)}&z_{2}^{(k)}\ar@{.}[rr]&&z_{n}^{(k)}
        }
        \POS"Ma1,1"."Ma4,5"!C*\frm{(}*\frm{)},
        \POS(55,0)
        \xymatrix"Mb"@R=.5pt@C=.5pt@W=1em@H=1em{
          \frac{\mathsf{D}_{t}^{1}(t)}{1!}&
          \frac{\mathsf{D}_{t}^{1}(z_{1})}{1!}&
          \frac{\mathsf{D}_{t}^{1}(z_{2})}{1!}\ar@{.}[rr]&
          &
          \frac{\mathsf{D}_{t}^{1}(z_{n})}{1!}\\
          \frac{\mathsf{D}_{t}^{2}(t)}{2!}\ar@{.}[dd]&
          \frac{\mathsf{D}_{t}^{2}(z_{1})}{2!}\ar@{.}[dd]&
          \frac{\mathsf{D}_{t}^{2}(z_{2})}{2!}\ar@{.}[dd]\ar@{.}[rr]&
          &
          \frac{\mathsf{D}_{t}^{2}(z_{n})}{2!}\ar@{.}[dd]\\
          \\
          \frac{\mathsf{D}_{t}^{k}(t)}{k!}&
          \frac{\mathsf{D}_{t}^{k}(z_{1})}{k!}&
          \frac{\mathsf{D}_{t}^{k}(z_{2})}{k!}\ar@{.}[rr]&
          &
          \frac{\mathsf{D}_{t}^{k}(z_{n})}{k!}
        }
        \POS"Mb1,1"."Mb4,5"!C*\frm{(}*\frm{)},
        +L\ar@{}|{\displaystyle=\qquad}"Ma1,1"."Ma4,5"!C+R,
        \POS(0,-35)
        \xymatrix"Mc"@R=.5pt@C=.5pt@W=1.25em@H=1em{
          t^{[1]}&1&z_{2}^{[1]}\ar@{.}[rr]&&z_{n}^{[1]}\\
          t^{[2]}\ar@{.}[dd]&0\ar@{.}[dd]&z_{2}^{[2]}\ar@{.}[rr]\ar@{.}[dd]&&z_{n}^{[2]}\ar@{.}[dd]\\
          &&&&\\
          t^{[k]}&0&z_{2}^{[k]}\ar@{.}[rr]&&z_{n}^{[k]}
        }
        \POS"Mc1,1"."Mc4,5"!C*\frm{(}*\frm{)},
        +L*++!R\txt{\(\mathbi{B}(z_{1})\)},
        \POS"Mc1,1"."Mc4,5"!C+U\ar@{}|{\rotatebox{-90}{$\displaystyle=$}}"Ma1,1"."Ma4,5"!C+D,
        \POS(55,-35)
        \xymatrix"Md"@R=.5pt@C=.5pt@W=1em@H=1em{
          \frac{\mathsf{D}_{z_{1}}^{1}(t)}{1!}&
          \frac{\mathsf{D}_{z_{1}}^{1}(z_{1})}{1!}&
          \frac{\mathsf{D}_{z_{1}}^{1}(z_{2})}{1!}\ar@{.}[rr]&
          &
          \frac{\mathsf{D}_{z_{1}}^{1}(z_{n})}{1!}\\
          \frac{\mathsf{D}_{z_{1}}^{2}(t)}{2!}\ar@{.}[dd]&
          \frac{\mathsf{D}_{z_{1}}^{2}(z_{1})}{2!}\ar@{.}[dd]&
          \frac{\mathsf{D}_{z_{1}}^{2}(z_{2})}{2!}\ar@{.}[dd]\ar@{.}[rr]&
          &
          \frac{\mathsf{D}_{z_{1}}^{2}(z_{n})}{2!}\ar@{.}[dd]\\
          \\
          \frac{\mathsf{D}_{z_{1}}^{k}(t)}{k!}&
          \frac{\mathsf{D}_{z_{1}}^{k}(z_{1})}{k!}&
          \frac{\mathsf{D}_{z_{1}}^{k}(z_{2})}{k!}\ar@{.}[rr]&
          &
          \frac{\mathsf{D}_{z_{1}}^{k}(z_{n})}{k!}
        }
        \POS"Md1,1"."Md4,5"!C*\frm{(}*\frm{)},+L*++!R\txt{\(\mathbi{B}(z_{1})\)},
        \ar@{}|{\displaystyle=\qquad}"Mc1,1"."Mc4,5"!C+R,
        \POS"Md1,1"."Md4,5"!C+U\ar@{}|{\rotatebox{-90}{$\displaystyle=$}}"Mb1,1"."Mb4,5"!C+D,
    \endxy},
  \]
the right column equality being obtained by applying the combinatorial relations provided by \eqref{cor:DtD1}. The invertibility of \(\mathbi{B}(z_{1})\) on \(\{z_{1}'\neq0\}\) allows to conclude.
\end{proof}

Besides the chain rule \eqref{eq:chainrule} that we have used, there is another natural definition for \(\mathsf{D}_{z_{1}}\). We use the convention  \(\vf[t]{}{}\bydef z_{1}'\vf{1}{}\). Then a natural generalization of the total differentiation in the new jet coordinate system is:
\[
  \left(
  \sum_{p=0}^{k-1}
  (p+1)\,
  t^{[p+1]}\;
  \vf[t]{}{[p]}
  \right)
  +
  \sum_{i=2,\dotsc,n}
  \left(
  \sum_{p=0}^{k-1}
  (p+1)\,
  z_{i}^{[p+1]}\;
  \vf{i}{[p]}
  \right).
\]

A computation shows that expressed in the standard jet coordinate system, this vector field coincides with \(\mathsf{D}_{z_{1}}\) on the vector space spanned by \(\{\vf{i}{(p)}\}_{p\leq k-1}\) \emph{but} has non zero terms in the vector vector space spanned by \(\{\vf{i}{(k)}\}_{i=1,\dotsc,n}\).
However, as long as we consider iterations \(\mathsf{D}_{z_{1}}^{p}\) acting on \(\C[z_{1},\dotsc,z_{n}]\), with \(p\leq k\), any of the two fields can be used for \(\mathsf{D}_{z_{1}}\) because it is never applied to \(k\)-th order terms.

Further observing that, by the above lemma, the formal derivatives \(\mathsf{D}_{z_{1}}\cdot z_{j}\), expressed in the geometric jet coordinate system, do not depend on the jet coordinates \(t^{[1]},\dotsc,t^{[k]}\), thus for any polynomial \(P\in\C[z_{1},\dotsc,z_{n}]\) nor do \(\mathsf{D}_{z_{1}}^{p}\cdot P\), we obtain the following (notice that \(t^{[1]}\vf[t]{}{}=\vf{1}{}\)).

\begin{lemma}
  While \(p\leq k\), the differential operator \(\mathsf{D}_{z_{1}}^{p}\) has the same action on polynomials \(P(z_{1},\dotsc,z_{n})\) as the differential operator:
  \[
    \left(
    \vf{1}{}
    +
    \sum_{i=2,\dotsc,n}
    \left(
    \sum_{p=0}^{k-1}
    (p+1)\,
    z_{i}^{[p+1]}\;
    \vf{i}{[p]}
    \right)
    \right)^{p}.
  \]
\end{lemma}
We will now take advantage of the very simple expression of the \(z_{1}\)-component of this vector field.

\section{Construction of Slanted Vector Fields}
\label{se:construction}
In this section, we work in the compact case, for \(c=1\).
\subsubsection{Vertical jets.}
Recall that the universal hypersurface is the subspace \(\mathcal{H}_{d}\subset \P^{n+1}\times \abs{\mathcal{O}(d)}\)
defined by:
\[
  \mathcal{H}_{d}
  \colon
  \sum_{\abs{\alpha}=d}
  A_{\alpha}\,Z^{\alpha}
  =
  0.
\]
\label{sse:verticalJets}
According to Siu, a \textsl{vertical \(k\)-jet} of the universal hypersurface \(\mathcal{H}_{d}\to \abs{\mathcal{O}(d)}\) is a \(k\)-jet of \(\mathcal{H}_{d}\) representable by some complex curve germ lying completely in some fiber \(H(A)\) over a certain point \(A\in \abs{\mathcal{O}(d)}\) of the parameter space.
Concretely, the formal differentiation presentation of Merker (\cite{MR2543663}) provides an efficient description of the subspace of vertical jets, recalled just below.

We restrict to the affine set \(\{Z_{0}\neq0\}\simeq\C^{n}\) 
equipped with the standard inhomogeneous coordinates
\(
(z_{1},\dotsc,z_{n})
\),
where as usual \(z_{j}=Z_{j}/Z_{0}\),
and for some \(\widehat{\alpha}\in\N^{n+1}\) with \(\abs{\widehat\alpha}=d\) and \(\widehat\alpha_{0}= 0\) (fixed later) we also restrict to the affine set \(\{A_{\widehat{\alpha}}\neq 0\}\simeq\C^{n_{d}}\) (here \(n_{d}=\binom{n+d}{n}-1\)),
equipped with the standard inhomogeneous coordinates:
\[
  a_{\alpha_{1},\dotsc,\alpha_{n}}
  \bydef
  A_{(d-\alpha_1-\dotsb-\alpha_n,\alpha_1,\dotsc,\alpha_n)}/A_{\widehat{\alpha}}
  \quad
  {\scriptstyle(\abs{\alpha}\leq d)}.
\] 
Notice that \(\alpha_{0}\) does not appear anymore in the indices of \(a_{\alpha}\).
For brevity, we will make the convention \(a_{\widehat{\alpha}}=a_{\widehat{\alpha}_{1},\dotsc,\widehat{\alpha}_{n}}=1\)
but keep in mind that \(a_{\widehat{\alpha}}\) is a constant; in particular there is no associated vector field \(\vf[a]{\widehat{\alpha}}{}\).

In these coordinate system, the restriction to \(U_{0}\bydef\{Z_{0}\neq0\}\cap\{A_{\widehat{\alpha}}\neq0\}\) of \(\mathcal{H}_{d}\) is the zero set:
\[
  \mathcal{Z}_{0}
  \colon
  \sum_{\abs{\alpha}\leq d}
  a_{\alpha}\,z_{1}^{\alpha_{1}}\dotsm z_{n}^{\alpha_{n}}
  =
  0.
\]

Considering the coefficients \(a_{\alpha}\) as independent variables, we work on the jet-space:
\[
  \C^{n_{d}}\times J_{k}(\C^{n})
  \simeq
  \C^{n_{d}}
  \times
  \C^{n}
  \times
  \C^{nk},
\]
equipped with \(a,z\) and the standard jet-coordinates \(z',\dotsc,z^{(k)}\). Then, for \(p=0,1,\dotsc,k\), let the expression \((\mathsf{D}_{t}^{p}\cdot P)\) stands for the polynomial in the variables \(a,z,z',\dotsc,z^{(k)}\) of the ambient jet-space
\(
\C^{n_{d}}
\times
\C^{n}
\times
\C^{nk}
\)
obtained by applying \(p\) times the formal differentiation 
\(
\mathsf{D}_{t}
\) 
to the defining equation \(P\) of \(\mathcal{Z}_{0}\).

The space \(J_{k}^{\mathit{vert}}(\mathcal{Z}_{0})\) of vertical \(k\)-jets consists of the \(k\)-jets of \(\mathcal{Z}_{0}\) satisfying the \(k+1\) equations:
\begin{equation}
  \label{eq:DP}
  0=
  P=
  \mathsf{D}_{t}\cdot P=
  \dotsb=
  \mathsf{D}_{t}^{k}\cdot P.
\end{equation}
Notice that these algebraic equations form a linearly free system of rank \((k+1)\), hence the algebraic subspace of logarithmic vertical jets \(J_{k}^{\mathit{vert}}(\mathcal{Z}_{0})\) is of pure codimension \((k+1)\) in the ambient space \(\C^{n_{d}}\times\C^{n\,(k+1)}\). 

In what follows, for brevity, \(J_{k}\) will stand for the ambient jet space \(\C^{n_{d}}\times J_{k}(\C^{n})\) and 
\(J_{k}^{\mathit{vert}}\) will stand for \(J_{k}^{\mathit{vert}}(\mathcal{Z}_{0})\).

\subsubsection{Algebraic equations of the tangent space.}
We will consider two types of vector fields on \(J_{k}^{\mathit{vert}}\):
\begin{itemize}
  \item 
    \textsl{Vertical vector fields} that are vector fields on \(J_{k}^{\mathit{vert}}\) tangential to the space of \(k\)-jets of the vertical fiber of \(J_{k}\to\C^{n_{d}}\).
  \item 
    \textsl{Slanted vector field} that are, according to Siu, vector fields on \(J_{k}^{\mathit{vert}}\) that are not tangential to the space of \(k\)-jets of the vertical fiber at a generic point of \(J_{k}^{\mathit{vert}}\).
\end{itemize}
We will show that there are very few vector fields of the first kind, what justifies to consider vector fields of the second kind.
Indeed, in order to prove global generation, the problem becomes the following: for every point \(p\in U_{0}\), construct enough slanted vector fields so that together with the vertical vector fields they form a (large) subspace of codimension \((k+1)\) in the vector space \(J_{k}^{\mathit{vert}}\rvert_{p}\).

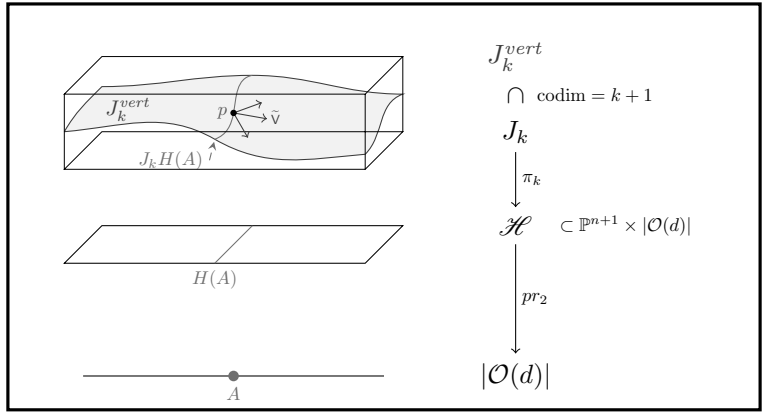
\begin{center}
  \begin{tikzpicture}
    \draw[very thick] (-1,-1.2) rectangle (9,4.2);
    \coordinate (C) at (2,-.75);
    \draw (C) +(-2,0)--+(2,0);
    \fill[DarkOrange!70!black] (C) circle[radius=2pt] node[below=2,scale=.7]{$A$};
    \coordinate (H) at (-.25,.75);
    \draw (H)--+(4,0)--+(4.5,.5)--+(.5,.5)--cycle;
    \draw[DarkOrange!70!black,thin] (H) +(2,0) node[below,scale=.7]{$H(A)$}--+(2.5,.5);
    \coordinate (J) at (-.25,2); 
    \path (J)++(2.25,.75) node[coordinate](o){};
    \node at (o) [left,scale=.7]{$p$};
    \path (J)++(2,.4) node[coordinate](a){};
    \path (J)++(2.5,1.25) node[coordinate](b){};
    \path (J)++(0,.5) node[coordinate](a-){};
    \path (J)++(.5,1.1) node[coordinate](b-){};
    \path (J)++(4,.2) node[coordinate](a+){};
    \path (J)++(4.5,1) node[coordinate](b+){};
    \filldraw[DarkGreen,fill=smartgreen,fill opacity=.4] (a-) to[out=5,in=150] (a) to[out=-30,in=185] (a+) to[out=45,in=-160] (b+) to[out=180,in=0] (b) to[out=180,in=0] (b-) to[out=-135,in=60] (a-);
    \draw[DarkOrange!70!black] (a) to[out=20,in=-100] (o) to[out=80,in=-170] (b);
    \node (lab) at (a) [DarkOrange!70!black,below left=2,scale=.7]{$J_{k}H(A)$};
    \draw[thin,DarkOrange!70!black,dashed,->,>=stealth,shorten >=1pt] (lab.east)--(a);
    \path(o)+(-1.4,0.05) node[DarkGreen,scale=.8]{$J_{k}^{\mathit{vert}}$};
    \draw[ultra thin] (J)--+(4,0)--+(4.5,.5)--+(.5,.5)--cycle +(0,1)--+(4,1)--+(4.5,1.5)--+(.5,1.5)--cycle (J)--+(0,1) +(4,0)--+(4,1) +(4.5,.5)--+(4.5,1.5);
    \draw[DarkGreen,->](o)--+(20:.4);
    \draw[DarkGreen,->] (o)--+(-10:.45) node[thick,near end,right=1mm,scale=.5]{$\widetilde{\mathsf{V}}$};
    \draw[DarkGreen,->] (o)--+(-60:.38);
    \filldraw (o) circle[radius=1pt];
    \path (J)+(6,.5) node (1) {$J_{k}$};
    \path (H)+(6,.5) node (2) {$\mathscr{H}$} node[right=5mm,scale=.7]{$\subset\P^{n+1}\times\abs{\mathcal{O}(d)}$};
    \path (C)+(3.75,0) node (3) {$\lvert\mathcal{O}(d)\rvert$};
    \draw[->](1)--(2) node[midway,right,scale=.7]{$\pi_k$};
    \draw[->](2)--(3) node[midway,right,scale=.7]{$pr_{2}$};
    \path (1)+(0,1)node[DarkGreen](4){$J_{k}^{\mathit{vert}}$};
    \path (4)--(1) node[midway,sloped]{$\subset$}node[midway,right=2mm,scale=.7]{$\mathrm{codim}=k+1$};
  \end{tikzpicture}
\end{center}
We will decompose a general vector field \(\widetilde{\mathsf{V}}\) on \(J_{k}\), that we will assume tangent to \(J_{k}^{\mathit{vert}}\), into two parts:
one vertical part \(\mathsf{V}\) tangent to the jet space \(J_{k}H(A)\) at points \(p\) with \(\pi_{k}\circ \mathrm{pr}_{2}(p)=A\) (which is the situation of the above picture) and one horizontal part \(\mathsf{U}\) tangent to the pullback of the space of parameters \(\abs{\mathcal{O}(d)}\):
\[
  \widetilde{\mathsf{V}}
  =
  \mathsf{V}-\mathsf{U}.
\]
Of course here neither of the two parts is assumed to be itself tangent to \(J_{k}^{\mathit{vert}}\).

For such a vector field \(\widetilde{\mathsf{V}}\) to be indeed tangential
to the submanifold \(J_{k}^{\mathit{vert}}\subset J_{k}\) of the ambient jet space,
the Lie derivatives along \(\widetilde{\mathsf{V}}\) of the \(k+1\) defining equations \eqref{eq:DP} have to be all zero at points of \(J_{k}^{\mathit{vert}}\):
\begin{equation}
  \label{eq:VDP}
  \widetilde{\mathsf{V}}
  \cdot
  \big(
  \mathsf{D}_{t}^{p}
  \cdot
  P
  \big)
  \big\vert_{J_{k}^{\mathit{vert}}}
  \equiv
  0
  \qquad
  {\scriptstyle(p=0,1,\dotsc,k)},
\end{equation}
where as above \(P\) is the universal polynomial
\(
P
=
\sum_{\lvert\alpha\rvert\leq d}
a_{\alpha}\,
z^{\alpha}
\).
The \textsl{Hadamard's lemma} asserts that a function vanishes identically on a submanifold if and only if it is locally a linear combination of the defining functions of this submanifold, whence
the system \eqref{eq:VDP} is equivalent to the existence for every point \(x\in J_{k}^{\mathit{vert}}\) of an open neighbourhood \(U_{x,k}\ni x\) and of functions \(H_{0}^{p},\dotsc,H_{k}^{p}\in\C^{U_{x,k}}\) such that:
\begin{equation}
  \label{eq:hadamard}
  \widetilde{\mathsf{V}}
  \cdot
  \big(
  \mathsf{D}_{t}^{p}
  \cdot
  P
  \big)
  \big\vert_{U_{x,k}}
  =
  H_{0}^{p}\,P
  +
  H_{1}^{p}\,
  \big(\mathsf{D}_{t}^{1}\cdot P\big)
  +
  \dotsb
  +
  H_{k}^{p}\,
  \big(\mathsf{D}_{t}^{k}\cdot P\big)
  \qquad
  {\scriptstyle(p=0,1,\dotsc,k)}.
\end{equation}

This latter characterization makes appear that the sheaves \(T_{J_{k}^{\mathit{vert}}},\dotsc, T_{J_{1}^{vert}}\) are \emph{not} compatible with the forgetful maps \(J_{k}\to J_{l}\) valued in lower order jet spaces, since when \(\widetilde{V}\) is a section of \(T_{J_{k}^{\mathit{vert}}}\), then the projection of \(\widetilde{V}\) is not necessarily a section of \(T_{J_{l}^\mathit{vert}}\): this only happens when \(H_{q}^{p}=0\) for \(p\leq l <q\).
For this reason, and since iterations of derivatives always have triangular properties (as shown in the Faà di Bruno formula \eqref{eq:faa1} above),
it is more natural to work with
the subsheaf of vector spaces \(\mathcal{T}_{k}\) of the tangent sheaf to \(J_{k}^{\mathit{vert}}\), whose sections are the vector fields that are tangential to \emph{each} submanifold
\[
  \mathcal{Z}_{p}
  \bydef
  \big\{
    P=\mathsf{D}_{t}\cdot P=\dotsb=\mathsf{D}_{t}^{p}\cdot P=0
  \big\}
  \quad
  {\scriptstyle(p=0,1,\dotsc,k)},
\]
not only at points of \(J_{k}^{\mathit{vert}}=\mathcal{Z}_{k}\subset \mathcal{Z}_{p}\), but at every point of \(\mathcal{Z}_{p}\), for all \(p\).

This subsheaf \(\mathcal{T}_{k}\) compares favourably with \(T_{J_{k}^{\mathit{vert}}}\) regarding the properties mentioned above.
Indeed, by Hadamard's lemma, the sections \(\widetilde{\mathsf{V}}\) of \(\mathcal{T}_{k}\) are characterized by the existence for each \(p=0,1,\dotsc,k\) and for every point \(x\in \mathcal{Z}_{p}\) of an open neighbourhood \(U_{x,p}\ni x\) and of functions \(H_{0}^{q},\dotsc,H_{p}^{q}\in\C^{U_{x,p}}\) such that:
\[
  \widetilde{\mathsf{V}}
  \cdot
  \big(
  \mathsf{D}_{t}^{q}
  \cdot
  P
  \big)
  \big\vert_{U_{x,p}}
  =
  H_{0}^{q}\,P
  +
  H_{1}^{q}\,
  \big(\mathsf{D}_{t}^{1}\cdot P\big)
  +
  \dotsb
  +
  H_{p}^{q}\,
  \big(\mathsf{D}_{t}^{p}\cdot P\big)
  \qquad
  {\scriptstyle(q=0,1,\dotsc,p)}.
\]
But fixing \(q\) instead of \(p\), for \(p\geq q\), because \(x\in \mathcal{Z}_{p}\subset \mathcal{Z}_{q}\), up to shrinking \(U_{x,p}\) so that \(U_{x,p}\subset U_{x,q}\), one has the more precise statement:
\[
  \widetilde{\mathsf{V}}
  \cdot
  \big(
  \mathsf{D}_{t}^{q}
  \cdot
  P
  \big)
  \big\vert_{U_{x,p}}
  =
  H_{0}^{q}\,P
  +
  H_{1}^{q}\,
  \big(\mathsf{D}_{t}^{1}\cdot P\big)
  +
  \dotsb
  +
  H_{q}^{q}\,
  \big(\mathsf{D}_{t}^{q}\cdot P\big);
\]
in other words, it is always possible to take \(H_{p}^{q}=0\) for \(q<p\).
To sum up, \(\mathsf{\widetilde{V}}\) is a section of \(\mathcal{T}_{k}\) if and only if for every \(p=0,1,\dotsc,k\), and every \(x\in \mathcal{Z}_{p}\) there exists an open neighbourhood \(U_{x,p}\) and a lower \emph{triangular} matrix \(H\in \mathrm{Mat}_{p}(\C^{U_{x,p}})\) such that:
\begin{equation}
  \label{eq:hadamard2}
  \widetilde{\mathsf{V}}
  \cdot
  \big(
  \mathsf{D}_{t}^{q}
  \cdot
  P
  \big)
  \big\vert_{U_{x,p}}
  =
  H_{0}^{q}\,P
  +
  H_{1}^{q}\,
  \big(\mathsf{D}_{t}^{1}\cdot P\big)
  +
  \dotsb
  +
  H_{q}^{q}\,
  \big(\mathsf{D}_{t}^{q}\cdot P\big)
  \qquad
  {\scriptstyle(q=0,1,\dotsc,p)}.
\end{equation}

Note that in the compact case we are dealing with here, this subtleties can be forgotten, since we will even find vector fields satisfying the equations \eqref{eq:VDP} not only on \(J_{k}^{\mathit{vert}}\) but identically on \(J_{k}\), \textit{i.e.} take \(H=0\). In the genuine logarithmic case however, we will need to restrict to \(\mathcal{Z}_{p}\) in order to cancel the \(p\)-th equation (see below).

\subsubsection{Fundamental case.}
For the sake of clear presentation of ideas, we will first consider the \emph{simplest case} where the vertical part \(\mathsf{V}\) is one of the vector fields \(\vf{j}{}\) for \(j=1,\dotsc,n\). The consideration of this case will prove to be fundamental to gain an intuition of the general solution to our problem.

The considered vertical vector field \(\mathsf{V}\bydef\vf{j}{}\) is \emph{not} tangent to the subspace of vertical jets, since already on the first defining equation in \eqref{eq:VDP} it is not zero:
\[
  \mathsf{V}\cdot P
  =
  \vf{j}{}
  \cdot
  \sum_{\abs{\alpha}\leq d}
  a_{\alpha}\,z^{\alpha}
  =
  \sum_{\abs{\alpha}\leq d}
  a_{\alpha}\,
  \alpha_{j}\,
  z^{\alpha-\mathbi{1}_{j}}
  =
  \sum_{\abs{\beta}\leq d-1}
  {\color{black!50}
  \underset{\hfill\scalebox{.7}{\framebox[1.5\width][c]{\color{black!80}$\defby u_{j,\beta}$}}}
  {\underline{\textcolor{black}{
    a_{\beta+\mathbi{1}_{j}}\,
    (\beta_{j}+1)
  }}}}\,
  z^{\beta}.
\]
But it is easy to cancel this last polynomial term \(\sum u_{j,\beta}\,z^{\beta}\) by means of an appropriate correction \(\mathsf{V}\mapsto\mathsf{V}-\mathsf{U}\). 
Indeed, to find such a vector field \(\mathsf{U}\), noticing that:
\[
  \vf[a]{\beta}{}
  \cdot P
  =
  \vf[a]{\beta}{}
  \cdot
  \left(
  {\textstyle \sum_{\abs{\alpha}\leq d} }\,
  a_{\alpha}\,z^{\alpha}
  \right)
  =
  z^{\beta},
\]
it suffices to take the following specific linear combination of the \(\\vf[a]{\beta}{}\) with coefficients precisely equal to the above \(u_{j,\beta}\):
\[
  \mathsf{U}
  \bydef
  \sum_{\abs{\beta}\leq d-1}
  u_{j,\beta}\,
  \vf[a]{\beta}{}
  =
  \sum_{\abs{\beta}\leq d-1}
  a_{\beta+{\bf 1}_{j}}\,
  (\beta_{j}+1)\,
  \vf[a]{\beta}{},
\]
which yields the sought vanishing of the first equation in \eqref{eq:VDP}, identically on \(J_{k}\):
\[
  \left(
  \mathsf{V}-\mathsf{U}
  \right)
  \cdot
  P
  =
  \bigg(
  \vf{j}{}
  -
  \sum_{\abs{\beta}\leq d-1}
  u_{j,\beta}\,
  \vf[a]{\beta}{}
  \bigg)
  \cdot
  P
  =
  \sum_{\abs{\beta}\leq d-1}\!
  u_{j,\beta}\,z^{\beta}
  -
  \sum_{\abs{\beta}\leq d-1}\!
  u_{j,\beta}\,z^{\beta}
  =
  0.
\]

So we get \(\widetilde{\mathsf{V}}\cdot P=(\mathsf{V}-\mathsf{U})\cdot P=0\). But recall that in order to prove the tangency to the subspace of vertical jets, it is also necessary to check the vanishing of the \(k\) remaining equations in \eqref{eq:VDP}: 
\[
  \widetilde{\mathsf{V}}
  \cdot 
  (
  \mathsf{D}_{t}^{q}
  \cdot 
  P
  )
  \stackrel?=
  0.
  \qquad
  {\scriptstyle(q=1,\dotsc,k)}.
\] 
What makes this first considered case very simple is that here having obtained \(\widetilde{\mathsf{V}}\cdot P=0\) actually suffices to conclude, because the coefficients of the constructed vector field:
\[
  \widetilde{\mathsf{V}}
  =
  \mathsf{V}-\mathsf{U}
  \bydef
  \vf{j}{}
  -
  \sum_{\abs{\beta}\leq d-1}
  a_{\beta+\mathbi{1}_{j}}\,(\beta_{j}+1)\,
  \vf[a]{\beta}{}
\]
do not depend on the \(z\)-variables, and thus the powers of the action \(\mathsf{D}_{t}\cdot\) and the action \(\widetilde{\mathsf{V}}\cdot\) commute,
whence:
\[
  \widetilde{\mathsf{V}}\cdot(\mathsf{D}_{t}^{q}\cdot P)
  =
  \mathsf{D}_{t}^{q}\cdot(\widetilde{\mathsf{V}}\cdot P)
  =
  0.
\]

\medskip

\begin{slshape}
  It should immediately be pointed out that \emph{the simplicity of the above computations could appear misleading}, since for a general vertical vector field:
  \[
    \mathsf{V}
    \bydef
    \sum_{i=1,\dotsc,n}
    \left(
    \sum_{q=0}^{k}
    v_{i}^{(q)}\!\big(a_{\alpha};\,z_{1},\dotsc,z_{n};\,z_{1}^{(1)},\dotsc,z_{n}^{(1)},\dotsc,z_{1}^{(k)},\dotsc,z_{n}^{(k)}\bigr)\,
    \vf{i}{(q)}
    \right),
  \]
  the huge linear systems one has to solve in order to find an adequate correction \(\mathsf{U}\) involves the multivariate Faà di Bruno formulas (\textit{cf.} \cite{MR2372741,MR2331545,MR2383820,MR2543663}), the expressions of which, though classical, are complicated.

  However, \emph{the above considerations are inspiring}:
  we will indeed avoid almost all technical inconveniences and make the general case nearly as simple as the simplest case.
\end{slshape}

\subsubsection{Strategy in the general case.}

We will proceed in three simplifying steps - \textit{a}, \textit{b} and \textit{c} - summarized here, the details of the two last being expanded in separate sections below.
\begin{enumerate}[label={\itshape\alph*.}]
  \item 
    Fix the vertical vector field \(\mathsf{V}\). Then find a \emph{correction \(\mathsf{U}\)} with the same image on the defining equations \eqref{eq:VDP}, and use it to cancel these equations:
    \[
      \label{eq:a} \tag{$a$}
      \mathsf{V}-\mathsf{U}
      \in
      \mathcal{T}_{k}
      \Longleftrightarrow
      \left\{
        \begin{array}{rclcrcll}
          \mathsf{V}
          &\cdot&
          P(\xi)
          &=&
          \mathsf{U}
          &\cdot& 
          P(\xi)
          &\big\vert_{\xi\in \mathcal{Z}_{0}}
          \\
          \mathsf{V}
          &\cdot&
          \big(
          \mathsf{D}_{t}
          \cdot 
          P
          \big)(\xi)
          &=&
          \mathsf{U}
          &\cdot&
          \big(
          \mathsf{D}_{t}
          \cdot 
          P
          \big)(\xi)
          &\big\vert_{\xi\in \mathcal{Z}_{1}}
          \\
          &&&\vdots&&&\\
          \mathsf{V}
          &\cdot&
          \big(
          \mathsf{D}_{t}^{k}
          \cdot 
          P
          \big)(\xi)
          &=&
          \mathsf{U}
          &\cdot&
          \big(
          \mathsf{D}_{t}^{k}
          \cdot 
          P
          \big)(\xi)
          &\big\vert_{\xi\in \mathcal{Z}_{k}},
        \end{array}
      \right.
    \]
    where as above \(\mathcal{Z}_{p}=\{P=\mathsf{D}_{t}\cdot P=\dotsb=\mathsf{D}_{t}^{p}\cdot P=0\}\subset J_{k}\).
  \item 
    When \(\mathsf{D}_{t}\) and \(\widetilde{\mathsf{V}}\) commute like in the fundamental case, one should directly see that a lot of equations cancel.
    Also, one should avoid as much as possible expanding the expressions \(\mathsf{D}_{t}^{p}\cdot P\) in order to reduce the combinatorial complexity.
    With these goals in mind, it makes sense to let \(\mathsf{D}_{t}\) act on the vector field \(\widetilde{\mathsf{V}}\) rather than on the equation \(P\), using \emph{adjoint action} (see below), which will shortly give: 
    \[
      \label{eq:b}\tag{$b$}
      \mathsf{V}-\mathsf{U}
      \in
      \mathcal{T}_{k}
      \Longleftrightarrow
      \left\{
        \begin{array}{rclcrcl}
          \mathsf{V}
          &\cdot&
          P(\xi)
          &=&
          \mathsf{U}
          &\cdot&
          P(\xi)
          \big\vert_{\xi\in \mathcal{Z}_{0}}
          \\
          (\mathsf{D}_{t}\ad\mathsf{V})
          &\cdot&
          P(\xi)
          &=&
          (\mathsf{D}_{t}\ad\mathsf{U})
          &\cdot&
          P(\xi)
          \big\vert_{\xi\in \mathcal{Z}_{1}}
          \\
          &&&\vdots&&&\\
          (\mathsf{D}_{t}^{k}\ad\mathsf{V})
          &\cdot&
          P(\xi)
          &=&
          (\mathsf{D}_{t}^{k}\ad\mathsf{U})
          &\cdot&
          P(\xi)
          \big\vert_{\xi\in \mathcal{Z}_{k}},
        \end{array}
      \right.
      .
    \]
    Making furthermore the choice \(\mathsf{V}=\vf{j}{(q)}\), which generalizes the fundamental case where \(\mathsf{V}=\vf{j}{}\) (\textit{i.e.} \(q=0\)), we will see that the values \((\mathsf{D}_{t}^{p}\ad\vf{j}{(q)})\cdot P(\xi)\) one has to cancel become very easy to compute and that they do not depend on the jet variables.
  \item 
    To find a correction \(\mathsf{U}\) that cancels these polynomial values \((\mathsf{D}_{t}^{p}\ad\vf{j}{(q)})\cdot P(\xi)\), we claim that it suffices to determine first some \emph{building-block vector fields} \(\mathsf{U}\) ---\,having the same role as the vector fields \(\vf[a]{\beta}{}\) in the fundamental case (\(q=0\))\,--- such that the corresponding entries \((\mathsf{D}_{t}^{p}\ad\mathsf{U})\cdot P(\xi)\) in the last column of \eqref{eq:b} are all zero for \(p\neq q\), while it is the monomial \((\mathsf{D}_{t}^{q}\ad\mathsf{U})\cdot P(\xi)=z^{\beta}\) for \(p=q\). Indeed, by linearity with respect to the coefficients \(a_{\alpha}\), one can then use these building-blocks vector fields to piece together a corrective vector field.

    We will see later on that it is adequate to use the geometric jet coordinates described in \S\ref{sse:geometricJetCoordinates} in order to avoid expanding the expressions \(\mathsf{D}_{t}^{p}\ad\mathsf{U}\). 
    Over the subset of invertible jets, working without loss of generality on the set \(\{z_{1}'\neq0\}\), since \eqref{cor:DtD1} is triangular invertible, one can use \(\mathsf{D}_{z_{1}}\) instead of \(\mathsf{D}_{t}\) to define the vertical jets.
    Restarting and applying again \eqref{eq:a} and \eqref{eq:b}, the problem reduces in the end to the obtainment of building-block vector fields \(\mathsf{U}_{q}^{\beta}\) such that:
    \begin{equation*}
      \label{eq:c}\tag{$c$}
      \text{if }p\neq q\colon\quad
      (\mathsf{D}_{z_{1}}^{p}\ad\mathsf{U}_{q}^{\beta})\cdot P(\xi)=0
      \qquad
      \text{and}\colon\quad
      (\mathsf{D}_{z_{1}}^{q}\ad\mathsf{U}_{q}^{\beta})\cdot P(\xi)=z^{\beta}.
    \end{equation*}

\end{enumerate}

\subsubsection{Use adjoint action.}
According to the strategy outlined just above, we consider the \textsl{Lie derivative of the vector field \(\widetilde{\mathsf{V}}\) along \(\mathsf{D}_{t}\)}, that is by definition the vector field \((\mathsf{D}_{t}\ad \widetilde{\mathsf{V}})\) acting on functions as:
\begin{equation}
  \label{eq:ad}
  (\mathsf{D}_{t}\ad\widetilde{\mathsf{V}})\cdot \bullet
  \bydef
  \mathsf{D}_{t}\cdot(\widetilde{\mathsf{V}}\cdot\bullet)
  -
  \widetilde{\mathsf{V}}\cdot(\mathsf{D}_{t}\cdot \bullet).
\end{equation}
For the sake of clarity, we will denote the adjoint action by ``\(\ad\)'' in order to alert to the fact, also emphasized by the use of parentheses in \eqref{eq:ad} and throughout this text, that a unified notation for the action of vector fields on functions and on vector fields would not be associative, what is already apparent in the definition \eqref{eq:ad}.
As an example, \(\widetilde{\mathsf{V}}\cdot P\equiv 0\) does not imply \((\mathsf{D}_{t}\ad\widetilde{\mathsf{V}})\cdot P\equiv0\).

\begin{lemma}
  \label{lem:VDP-LVP}
  The vector field \(\widetilde{\mathsf{V}}\) is a section of \(\mathcal{T}_{k}\) 
  if and only if:
  \[
    \bigl(
    \mathsf{D}_{t}^{p}
    \ad
    \widetilde{\mathsf{V}}
    \big)
    \cdot
    P
    \big\vert_{
      \{
        P=\mathsf{D}_{t}\cdot P=\dotsb=\mathsf{D}_{t}^{p}\cdot P=0
      \}
    }
    \equiv
    0
    \qquad
    {\scriptstyle(p=0,1,\dotsc,k)}.
  \]
\end{lemma}
\begin{proof}
  Using again Hadamard's lemma, the function \(\big(\mathsf{D}_{t}^{p}\ad\widetilde{\mathsf{V}}\big)\cdot P\) vanishes identically on the submanifold 
  \(\mathcal{Z}_{p}
  \bydef
  \big\{
    P=\mathsf{D}_{t}\cdot P=\dotsb=\mathsf{D}_{t}^{p}\cdot P=0
  \big\}
  \)
  if and only if for every point \(x\in \mathcal{Z}_{p}\)
  there exists an open neighbourhood \(V_{x,p}\) and functions \(G_{p}^{q}\in\C^{V_{x,p}}\) such that:
  \[
    \big(
    \mathsf{D}_{t}^{q}
    \ad
    \widetilde{\mathsf{V}}
    \big)
    \cdot
    P
    \big\vert_{V_{x,p}}
    =
    G_{0}^{q}\,P
    +
    G_{1}^{q}\,
    \big(\mathsf{D}_{t}^{1}\cdot P\big)
    +
    \dotsb
    +
    G_{p}^{q}\,
    \big(\mathsf{D}_{t}^{q}\cdot P\big)
    \quad
    {\scriptstyle(q=0,1,\dotsc,p)}.
  \]

  Thus \(\widetilde{\mathsf{V}}\) fulfill to the \(k+1\) conditions of the statement if and only if for \(p=0,1,\dotsc,k\) and for every point \(x\in \mathcal{Z}_{p}=\mathcal{Z}_{0}\cap\dotsb\cap \mathcal{Z}_{p}\), their exists an open neighbourhood \(V_{x,p}'=V_{x,0}\cap\dotsb\cap V_{x,p}\) and a lower triangular matrix \(G\in \mathrm{Mat}_{k}(\C^{V_{x,p}'})\) such that:
  \[
    \big(
    \mathsf{D}_{t}^{q}
    \ad
    \widetilde{\mathsf{V}}
    \big)
    \cdot
    P
    \big\vert_{V_{x,p}}
    =
    G_{0}^{q}\,P
    +
    G_{1}^{q}\,
    \big(\mathsf{D}_{t}^{1}\cdot P\big)
    +
    \dotsb
    +
    G_{q}^{q}\,
    \big(\mathsf{D}_{t}^{q}\cdot P\big)
    \qquad
    {\scriptstyle(q=0,1,\dotsc,p)}.
    \tag{$\ast$}
  \]

  Now, from the very definition of the adjoint action \eqref{eq:ad}, one can deduce the combinatorial formulas:
  \begin{align*}
    (\mathsf{D}_{t}^{q}
    \ad
    \widetilde{\mathsf{V}})
    \cdot
    \bullet
    &=
    \sum_{p=0}^{q}
    (-1)^{p}
    \binom{q}{p}
    \mathsf{D}_{t}^{q-p}
    \cdot
    \Bigl(
    \widetilde{\mathsf{V}}
    \cdot
    (
    \mathsf{D}_{t}^{p}
    \cdot\bullet
    )
    \Bigr),
    \tag{$\Rightarrow$}
    \\
    \intertext{and inversely:}
    \widetilde{\mathsf{V}}\cdot
    (
    \mathsf{D}_{t}^{q}
    \cdot\bullet
    )
    &=
    \sum_{p=0}^{q}
    (-1)^{p}
    \binom{q}{p}
    \mathsf{D}_{t}^{q-p}
    \cdot
    \Bigl(
    (\mathsf{D}_{t}^{p}\ad\widetilde{\mathsf{V}})
    \cdot\bullet
    \Bigr).
    \tag{$\Leftarrow$}
  \end{align*}
  These formulas allow respectively to go from the first characterization \eqref{eq:hadamard} to the second one ($\ast$) and the other way around, since the Leibniz rule:
  \[
    \mathsf{D}_{t}\cdot(fg)
    =
    (\mathsf{D}_{t}\cdot f)\,g
    +
    f\,(\mathsf{D}_{t}\cdot g),
  \]
  implies that the \((q-p)\)-th derivative along \(\mathsf{D}_{t}\) of a linear combination of \(P\), \(\mathsf{D}_{t}^{1}\cdot P\), \dots, \(\mathsf{D}_{t}^{p}\cdot P\) is automatically a linear combination of \(P\), \(\mathsf{D}_{t}^{1}\cdot P\), \dots, \(\mathsf{D}_{t}^{q}\cdot P\).
\end{proof}
From now on, we will denote by \(\mathbi{e}_{0},\mathbi{e}_{1},\dotsc,\mathbi{e}_{k}\) the standard basis of units vectors for \(\C^{k+1}\).
We will also denote by \(\Lambda_{t}\) the linear map
associating to each vector field \(\widetilde{\mathsf{V}}\) on \(J_{k}\) a vector-valued symmetric form on \(J_{k}\):
\begin{equation}
  \label{eq:Lt}
  \Lambda_{t}\big(\widetilde{\mathsf{V}}\big)
  \bydef
  \bigl(\widetilde{\mathsf{V}}\cdot P\bigr)\;\mathbi{e}_{0}+
  \bigl(\,(\mathsf{D}_{t}^{1}\ad\widetilde{\mathsf{V}})\cdot P\,\bigr)\;\mathbi{e}_{1}+
  \dotsb+
  \bigl(\,(\mathsf{D}_{t}^{k}\ad\widetilde{\mathsf{V}})\cdot P\,\bigr)\;\mathbi{e}_{k},
\end{equation}
in such a way that the sections of \(\mathcal{T}_{k}\) are precisely those satisfying:
\[
  \Lambda_{t}(\widetilde{\mathsf{V}})
  \equiv 
  H (P,\mathsf{D}_{t}\cdot P,\dotsc,\mathsf{D}_{t}^{k}\cdot P),
\] 
with \(H\) lower triangular (by Lemma \ref{lem:VDP-LVP}).
Recall that in the compact case, we will even take \(H=0\).

Let us now explain how using the adjoint action simplifies the problem.
Using the \textsl{Leibniz rule} for the general Lie derivative:
\begin{equation}
  \label{eq:leibniz}
  \mathsf{V}_{1}\ad (f\,\mathsf{V}_{2})
  =
  (\mathsf{V}_{1}\cdot f)\,\mathsf{V}_{2}
  +
  f\,(\mathsf{V}_{1}\ad\mathsf{V}_{2}),
\end{equation}
the crucial observation of the fundamental case, namely that \(\mathsf{D}_{t}\ad \vf{j}{}=0\), can be generalized as follows:
\begin{lemma}
  \label{lem:Dtvf}
  For any index \(j=1,\dotsc,n\), one has \(\mathsf{D}_{t}\ad \vf{j}{}=0\) and for any order of derivation \(q=1,\dotsc,k\), the Lie derivative of the vector field \(\vf{j}{(q)}\) is the following plain vector field of the same kind:
  \[
    \mathsf{D}_{t}
    \ad
    \vf{j}{(q)}
    =
    -q\,
    \vf{j}{(q-1)}.
  \]
\end{lemma}
\begin{proof}
  Recall that by definition, the formal differentiation of \(k\)-jets is the linear map associated to the vector field
  \[
    \mathsf{D}_{t}
    \bydef
    \sum_{i=1,\dotsc,n}
    \biggl(\,
    \sum_{p=0}^{k-1}
    (p+1)\,
    z_{i}^{(p+1)}\,\vf{i}{(p)}
    \biggr).
  \]
  Now, it is might be more intuitive to use the antisymmetry in \eqref{eq:ad}:
  \[
    \mathsf{D}_{t}
    \ad
    \vf{j}{(q)}
    =
    -
    \vf{j}{(q)}
    \ad
    \mathsf{D}_{t},
  \] 
  in order to reshape the sought quantity.
  Next, applying the Leibniz rule \eqref{eq:leibniz}, one gets the expression:
  \[
    \vf{j}{(q)}
    \ad
    \mathsf{D}_{t}
    =
    \sum_{i=1,\dotsc,n}
    \biggl(\,
    \sum_{p=0}^{k-1}
    (p+1)\,
    \left(\vf{j}{(q)}\cdot z_{i}^{(p+1)}\right)
    \vf{i}{(p)}
    \biggr)
    +
    \sum_{i=1,\dotsc,n}
    \biggl(\,
    \sum_{p=0}^{k-1}
    (p+1)\,
    z_{i}^{(p+1)}\,
    \left(\vf{j}{(q)}\ad\vf{i}{(p)}\right)
    \biggr),
  \]
  in which all the terms but one in the first summand are zero, by independence of the jet variables, and all terms in the second summand are zero by commutativity.
  The only remaining term is \(q\,\vf{j}{(q-1)}\).
\end{proof}
The result of the above lemma \ref{lem:Dtvf} is easy to iterate, in order to obtain that for any integers \(0\leq p,q\leq k\) the \(p\)-th Lie derivative of the vector field \(\vf{j}{(q)}\) along \(\mathsf{D}_{t}\) is the vector field:
\[
  \mathsf{D}_{t}^{p}
  \ad
  \vf{j}{(q)}
  =
  (-1)^{p}\,
  \frac{q!}{(q-p)!}\,
  \vf{j}{(q-p)},
\]
where \(\vf{j}{(q-p)}\) is by convention \(0\) for \(p>q\).
Applied to the equation \(P\) --- that depends on the variables \(\{z_{1},\dotsc,z_{n}\}\) but not on the jet-coordinates \(\{z_{i}^{(p)}\}_{p\geq1}\) --- the vector field \(\vf{j}{(q-p)}\) always vanishes, unless \(q-p=0\), in which case it is \(\vf{j}{(0)}=\\vf{j}{}\).
\begin{corollary}
  \label{cor:Lvfjq}
  The vector field \(\vf{j}{(q)}\) does not satisfy \(\Lambda_{t}(\vf{j}{(q)})=0\), because it produces a non zero entry in the \((q+1)\)-th line:
  \[
    \xym{6,1}{
      \vf{j}{(q)}
      \cdot
      P
      \ar@{.}[ddd]\\
      \\
      \\
      (\mathsf{D}_{t}^{q}\ad\vf{j}{(q)})
      \cdot
      P
      \ar@{.}[dd]\\
      \\
      (\mathsf{D}_{t}^{k}\ad\vf{j}{(q)})
      \cdot
      P
    }
    \equiv
    \xym[.16]{10,1}{
      \hskip 5pt\vf{j}{(q)}\cdot P\hskip5pt\\
      \\\\
      \\
      \\
      q!\,\vf{j}{(0)}\cdot P\ar@{.}[uuuuu] \\
      0\ar@{.}[ddd] \\
      \\\\
      0
    }
    \equiv
    \xym[.21]{10,1}{
      \hskip 15pt 0 \hskip 15pt \\
      \\\\
      \\
      0\ar@{.}[uuuu] \\
      q!\,\vf{j}{}\cdot P \rlap{\quad \(\leftarrow\) line of{} \(\mathsf{D}_{t}^{q}\).} \\
      0\ar@{.}[ddd] \\
      \\\\
      0
    }
  \]
\end{corollary}
This situation generalizes the situation of the fundamental case (which is the case where \(q=0\)). 

The remaining difficulty is to compute the adequate compensation part \(\mathsf{U}\) for \(\mathsf{V}=\vf{j}{(q)}\), namely a vector field in the direction of the parameter space such that \(\Lambda_{t}(\mathsf{U})=\vf{j}{}\cdot P\;\mathbi{e}_{q}\), that is to say:
\[
  \text{if }p\neq q\colon\quad
  (\mathsf{D}_{t}^{p}\ad\mathsf{U})\cdot P(\xi)=0
  \qquad
  \text{and}\colon\quad
  (\mathsf{D}_{t}^{q}\ad\mathsf{U})\cdot P(\xi)=\vf{j}{}\cdot P(z).
\]

\subsubsection{Interlude. Origins of P\u{a}un's layout uncovered.}
For a short moment we will only treat the next case \(\mathsf{V}=\vf{j}{(1)}\), that is the case \(q=1\), because it already involves the main arguments of our strategy and also because it is an occasion to revisit the papers of P\u{a}un (\cite{MR2372741}), Rousseau (\cite{MR2331545,MR2383820,MR2552951}) and Merker (\cite{MR2543663}).

We have seen, as a particular case of Corollary \ref{cor:Lvfjq}, that:
\[
  \Lambda_{t}\big(\vf{j}{(1)}\big)
  =
  \bigl(\vf{j}{}\cdot P\bigr)\;\mathbi{e}_{1}.
\]
According to the strategy sketched above, one has to find a corrective vector field \(\mathsf{U}\) with the same image
\(
\Lambda_{t}(\mathsf{U})
=
(\vf{j}{}\cdot P)\;\mathbi{e}_{1}
\).
Inspired by the use of the vector fields \(\vf[a]{\beta}{}\) as building-block vector fields in the fundamental case, for each \(\beta\) with \(\abs{\beta}\leq d-1\), we will first seek for vector fields \(\mathsf{U}_{\scriptscriptstyle(1)}^{\beta}\) such that:
\begin{equation}
  \label{eq:LtU1}
  \Lambda_{t}\big(\mathsf{U}_{\scriptscriptstyle(1)}^{\beta}\big)
  =
  z^{\beta}\;\mathbi{e}_{1}.
\end{equation}
This allows indeed to conclude as in the fundamental case above,
because we have computed:
\[
  \vf{j}{}
  \cdot
  P
  =
  \sum_{\abs{\beta}\leq d-1}
  {\color{black!50}
  \underset{\hfill\scalebox{.7}{\framebox[1.5\width][c]{\color{black!80}$= u_{j,\beta}$}}}
  {\underline{\textcolor{black}{
    a_{\beta+\mathbi{1}_{j}}\,
    (\beta_{j}+1)
  }}}}\,
  z^{\beta},
\]
and thus, setting as expected \(\mathsf{U}=\sum\,u_{j,\beta}\,\mathsf{U}_{\scriptscriptstyle(1)}^{\beta}\), the only line that is not already zero by construction in \(\Lambda_{t}\big(\vf{j}{(1)}-\mathsf{U}\big)\) is cancelled, by linearity of \(\Lambda_{t}\) with respect to the coefficients \(u_{j,\beta}\):
\[
  \Lambda_{t}
  \bigg(
  \vf{j}{(1)}
  -
  \sum_{\abs{\beta}\leq d-1}
  u_{j,\beta}\,
  \mathsf{U}_{\scriptscriptstyle(1)}^{\beta}
  \bigg)
  =
  \bigg(
  \sum_{\abs{\beta}\leq d-1}\!
  u_{j,\beta}\,z^{\beta}
  -
  \sum_{\abs{\beta}\leq d-1}\!
  u_{j,\beta}\,z^{\beta}
  \bigg)\;
  \mathbi{e}_{1}
  =
  0.
\]

Now concerning \eqref{eq:LtU1}, for a fixed \(\beta\), the equation:
\[
  \Lambda_{t}\big(\mathsf{U}_{\scriptscriptstyle(1)}^{\beta}\big)
  \stackrel{\eqref{eq:Lt}}\bydef
  \bigl(\mathsf{U}_{\scriptscriptstyle(1)}^{\beta}\cdot P\bigr)\;\mathbi{e}_{0}+
  \bigl(\,(\mathsf{D}_{t}^{1}\ad\mathsf{U}_{\scriptscriptstyle(1)}^{\beta})\cdot P\,\bigr)\;\mathbi{e}_{1}+
  \dotsb+
  \bigl(\,(\mathsf{D}_{t}^{k}\ad\mathsf{U}_{\scriptscriptstyle(1)}^{\beta})\cdot P\,\bigr)\;\mathbi{e}_{k}
  =
  z^{\beta}\;\mathbi{e}_{1}
\]
is a higher order analog of:
\[
  \Lambda_{t}\big(\vf[a]{\beta}{}\big)
  \stackrel{\eqref{eq:Lt}}\bydef
  \bigl(\vf[a]{\beta}{}\cdot P\bigr)\;\mathbi{e}_{0}+
  \bigl(\,(\mathsf{D}_{t}^{1}\ad\vf[a]{\beta}{})\cdot P\,\bigr)\;\mathbi{e}_{1}+
  \dotsb+
  \bigl(\,(\mathsf{D}_{t}^{k}\ad\vf[a]{\beta}{})\cdot P\,\bigr)\;\mathbi{e}_{k}
  =
  z^{\beta}\,\mathbi{e}_{0}.
\]
A simple idea in order to produce the needed building-block vector field \(\mathsf{U}_{\scriptscriptstyle(1)}^{\beta}\) would hence be that \(\mathsf{U}_{\scriptscriptstyle(1)}^{\beta}\) satisfies only two properties:
\begin{equation}
  \label{eq:U1t}
  \mathsf{U}_{\scriptscriptstyle(1)}^{\beta}\cdot P 
  =
  0
  \qquad\text{and}\qquad
  (\mathsf{D}_{t}^{1}\ad\mathsf{U}_{\scriptscriptstyle(1)}^{\beta})
  =
  \vf[a]{\beta}{};
\end{equation}
indeed, it would then follow that for \(p=1,\dotsc,k\):
\[
  \bigl(\mathsf{D}_{t}^{p}\ad\mathsf{U}_{\scriptscriptstyle(1)}^{\beta}\bigr)\cdot P
  =
  \bigl(\mathsf{D}_{t}^{p-1}\ad(\mathsf{D}_{t}\ad\mathsf{U}_{\scriptscriptstyle(1)}^{\beta})\bigr)\cdot P
  =
  \bigl(\mathsf{D}_{t}^{p-1}\ad\vf[a]{\beta}{}\bigr)\cdot P,
\]
whence, as announced:
\[
    \xy
    \xymatrix"LUU"@R=.5pt@C=.5pt@W=1em@H=2em{
      \makebox[60pt][c]{$\mathsf{U}_{\scriptscriptstyle(1)}^{\beta}\cdot P$}
      &=& 
      0\;\\
      \makebox[60pt][c]{$(\mathsf{D}_{t}\ad \mathsf{U}_{\scriptscriptstyle(1)}^{\beta})\cdot P$}
      &=& 
      z^{\beta}\;\\
      \makebox[60pt][c]{$(\mathsf{D}_{t}^{2}\ad \mathsf{U}_{\scriptscriptstyle(1)}^{\beta})\cdot P$}
      \ar@{.}[dd]&=&\ar@{.}[dd]
      0\;\\
      &&\\
      \makebox[60pt][c]{$(\mathsf{D}_{t}^{k}\ad \mathsf{U}_{\scriptscriptstyle(1)}^{\beta})\cdot P$}
      &=& 
      0\;
    }
    \POS"LUU5,1"."LUU1,3"!C*++\frm{(}*\frm{)},
    \POS"LUU1,1"."LUU5,3"!C+L-<35pt,0pt>*\txt{$
      \Lambda_{t}\big(\mathsf{U}_{\scriptscriptstyle(1)}^{\beta}\big)
      \;=\;
    $}
    \POS(50,-5)
    \xymatrix"LU"@R=.5pt@C=.5pt@W=1em@H=2em{
      \makebox[70pt][c]{$\vf[a]{\beta}{}\cdot P$}
      &=& 
      z^{\beta}\;\\
      \makebox[70pt][c]{$(\mathsf{D}_{t}\ad\vf[a]{\beta}{})\cdot P$}
      \ar@{.}[dd]&=&\ar@{.}[dd]
      0\;\\
      &&\\
      \makebox[70pt][c]{$(\mathsf{D}_{t}^{k-1}\ad\vf[a]{\beta}{})\cdot P$}
      &=& 
      0\;\\
      \makebox[70pt][c]{$(\mathsf{D}_{t}^{k}\ad\vf[a]{\beta}{})\cdot P$}
      &=& 
      0\;
    }
    \POS"LU1,1"."LU5,3"!C*++\frm{(}*\frm{)},
    \POS"LU1,1"."LU5,3"!C+R+<35pt,0pt>*\txt{$
      \;=\;
      \Lambda_{t}\big(\vf[a]{\beta}{}\big).
    $}

    \POS"LU1,1"."LU1,3"!C*-\frm{--},
    \POS"LUU2,1"."LUU2,3"!C*-\frm{--},
    \POS"LU1,1"+L\ar@(l,r)"LUU2,3"+R,
    \POS"LU2,1"."LU2,3"!C*-\frm{--},
    \POS"LUU3,1"."LUU3,3"!C*-\frm{--},
    \POS"LU2,1"+L\ar@(l,r)"LUU3,3"+R,
    \POS"LU4,1"."LU4,3"!C*-\frm{--},
    \POS"LUU5,1"."LUU5,3"!C*-\frm{--},
    \POS"LU4,1"+L\ar@(l,r)"LUU5,3"+R,
  \endxy
\]

As simple as the problem \eqref{eq:U1t} may firstly appear, one has to admit after a moment of reflection that it is not so, because one has to eliminate the jet derivatives that inevitably appear when expanding
\begin{equation}
  \label{eq:recallDt}
  \mathsf{D}_{t}
  =
  \sum_{i=1,\dotsc,n}
  \left(
  \sum_{p=0}^{k-1}
  (p+1)\;
  z_{i}^{(p+1)}\;
  \vf{i}{(p)}
  \right).
\end{equation}
However, our efforts will be repaid: we are about to see how to reap the benefits from the strategy sketched above, using the flexibility in the choice of the generators \(\widetilde{\mathsf{V}}\) of the tangent space.
\begin{enumerate}
  \item
    Notice that by Lemma \ref{cor:DtD1} one can use \(\mathsf{D}_{z_{1}}\) instead of \(\mathsf{D}_{t}\) to define in the exact analogous way the subspace of vertical jets, working without loss of generality on the subset \(\{z_{1}'\neq0\}\) of the set of invertible jets, on which we have to prove global generation of the tangent space, since:
    \[
      \left\{
        \begin{array}{cl}
          P(z)&=0\\
          \mathsf{D}_{t}\cdot P(z)&=0\\
          \vdots\\
          \mathsf{D}_{t}^{k}\cdot P(z)&=0
        \end{array}
      \right.
      \stackrel{\ (z_{1}'\neq0)\ }\Longleftrightarrow
      \left\{
        \begin{array}{cl}
          P(z)&=0\\
          \mathsf{D}_{z_{1}}\cdot P(z)&=0\\
          \vdots\\
          \mathsf{D}_{z_{1}}^{k}\cdot P(z)&=0.
        \end{array}
      \right.
    \]
  \item
    Recall from \S\ref{sse:geometricJetCoordinates} that the use of geometric jet coordinates on \(\{z_{1}'\neq0\}\) simplifies substantially the formal differentiation of jets in the special direction \(z_{1}\);
    we have seen that \(\mathsf{D}_{z_{1}}\) has component in the \(z_{1}\)-direction plainly equal to \(\vf{1}{}\), more precisely:
    \[
      \mathsf{D}_{z_{1}}\vert_{J_{k-1}}
      =
      \vf{1}{}
      +
      \sum_{i=2,\dotsc,n}
      \left(
      \sum_{p=0}^{k-1}
      (p+1)\;
      z_{i}^{[p+1]}\;
      \vf{i}{[p]}
      \right),
    \] 
    \textit{cp.} with the expression \eqref{eq:recallDt} of \(\mathsf{D}_{t}\) just above.
\end{enumerate}
As a consequence of the first point, all what we have done can be do using \(\mathsf{D}_{z_{1}}\) instead of \(\mathsf{D}_{t}\), and as a consequence of the second point, the analog of the problem \eqref{eq:U1t} becomes much simpler (see below).
In other words, the choice of the vertical vector fields \(\mathsf{V}=\vf{i}{(1)}\) was not the second simplest choice, since it is more easy to treat the case where:
\[
  \mathsf{V}
  =
  \vf{i}{[1]}.
\]
This remark explains why these vector fields already appeared in the matricial approach introduced by P\u{a}un and further pushed by Rousseau and Merker (\textit{cf}. \eqref{eq:vfpaun} above for a translation formula).

Restarting and doing all the same reasoning with \(\mathsf{D}_{z_{1}}\) and squared brackets on \(\{z_{1}'\neq0\}\), 
we set:
\[
  \Lambda_{z_{1}}\big(\widetilde{\mathsf{V}}\big)
  \bydef
  \bigl(\widetilde{\mathsf{V}}\cdot P\bigr)\;\mathbi{e}_{0}+
  \bigl(\,(\mathsf{D}_{z_{1}}^{1}\ad\widetilde{\mathsf{V}})\cdot P\,\bigr)\;\mathbi{e}_{1}+
  \dotsb+
  \bigl(\,(\mathsf{D}_{z_{1}}^{k}\ad\widetilde{\mathsf{V}})\cdot P\,\bigr)\;\mathbi{e}_{k},
\]
and we get that (for \(j\neq 1\)) the vector field
\(\vf{j}{[1]}-\mathsf{U}\) is tangent to \(J_{k}^{\mathit{vert}}\) over the set \(\{z_{1}'\neq0\}\) if:
\[
  \Lambda_{z_{1}}(\mathsf{U})
  =
  \Lambda_{z_{1}}(\vf{j}{[1]})
  =
  (\vf{j}{}\cdot P(z))\;\mathbi{e}_{1}.
\]
So we construct building-block vector fields \(\mathsf{U}_{1}^{\beta}\) such that \(\Lambda_{z_{1}}(\mathsf{U}_{1}^{\beta})=z^{\beta}\;\mathbi{e}_{1}\) by solving:
\begin{equation}
  \label{eq:U1}
  \mathsf{U}_{1}^{\beta}\cdot P =0
  \qquad\text{and}\qquad
  \mathsf{D}_{z_{1}}\ad\mathsf{U}_{1}^{\beta}
  =
  \vf[a]{\beta}{},
\end{equation}
but now, using the direction \(z_{1}\), the problem becomes essentially trivial since for \(\abs{\beta}\leq d-1\) a simple solution to \eqref{eq:U1} is to take:
\begin{equation}
  \mathsf{U}_{1}^{\beta}
  \bydef
  z_{1}\,\vf[a]{\beta}{}
  -
  \vf[a]{\beta+\mathbi{1}_{1}}{}.
\end{equation}
and finally we obtain a vector field tangent to the vertical jets with vertical part \(\vf{i}{[1]}\) in exactly the same way as in the fundamental case, since:
\[
  \Lambda_{z_{1}}
  \bigg(
  \vf{j}{[1]}
  -
  \sum_{\abs{\beta}\leq d-1}
  u_{j,\beta}\,
  \mathsf{U}_{1}^{\beta}
  \bigg)
  =
  \left(
  \sum_{\abs{\beta}\leq d-1}\!
  u_{j,\beta}\,z^{\beta}
  -
  \sum_{\abs{\beta}\leq d-1}\!
  u_{j,\beta}\,z^{\beta}
  \right)\;
  \mathbi{e}_{1}
  =
  0.
\]

\subsubsection{Building-block vector fields.}
Recall that on the set \(\{z_{1}'\neq0\}\) the vector field
\(
\widetilde{\mathsf{V}}
=
\mathsf{V}
-
\mathsf{U}
\)
is a section of \(\mathcal{T}_{k}\) if and only if it satisfies the \(k+1\) conditions, for \(p=0,1,\dotsc,k\):
\[
  \bigg(
  P(\xi)=\mathsf{D}_{z_{1}}\cdot P(\xi)=\dotsb=\mathsf{D}_{z_{1}}^{p}\cdot P(\xi)=0
  \bigg)
  \Rightarrow
  \bigg(
  \Lambda_{z_{1}}
  (
  \mathsf{V}
  )
  (\xi)
  =
  \Lambda_{z_{1}}(
  \mathsf{U}
  )
  (\xi)
  \bigg).
\]
But in the compact case, we will not need the left assumption to obtain the right conclusion.

When working on \(\{z_{1}'\neq0\}\), 
we demand that \(\widehat{\alpha}_{1}=0\), for technical reasons, and we fix \(\widehat\alpha\) once for all.

Since by Corollary \ref{cor:Lvfjq} the vector field associated to the geometric jet coordinate \(z_{j}^{[q]}\) satisfy:
\begin{equation}
  \label{eq:D1VF}
  \Lambda_{z_{1}}(
  \vf{j}{[q]}
  )
  \equiv
  \bigl(\vf{j}{}\cdot P\bigr)\,\mathbi{e}_{q},
\end{equation}
according to the strategy sketched out above, it needs to be corrected by an adequate linear combination of building-block vector fields \(\mathsf{U}\) such that
\[
  \Lambda_{z_{1}}(
  \mathsf{U}
  )
  \equiv
  z^{\beta}\,\mathbi{e}_{q}.
\]

The iteration of the inductive step \eqref{eq:U1t} allows to produce most of the needed building-block vector fields:
\begin{lemma}
  \label{lem:Uq}
  For \(q\in\N\) and \(\beta\) with \(\abs{\beta}+q\leq d\) the vector field:
  \[
    \mathsf{U}_{q}^{\beta}
    \bydef
    \sum_{p=0}^{q}
    \frac{(-1)^{p}}{p!(q-p)!}\,
    z_{1}^{q-p}\;
    \vf[a]{\beta+p\,\mathbi{1}_{1}}{}
  \]
  is a solution to
  \(
  \Lambda_{z_{1}}(
  \mathsf{U}
  )
  \equiv
  z^{\beta}\,\mathbi{e}_{q}
  \),
  where implicitly \(\mathbi{e}_{q}=\mathbi{0}\) if \(q>k\).
\end{lemma}
\begin{proof}
  One checks the inductive formula:
  \[
    \mathsf{U}_{q+1}^{\beta}
    =
    z_{1}\,\mathsf{U}_{q}^{\beta}
    -
    \mathsf{U}_{q}^{\beta+\mathbi{1}_{1}},
  \]
  whence, making an induction on \(q\), if \(\Lambda_{z_{1}}(\mathsf{U}_{q}^{\beta})=z^{\beta}\,\mathbi{e}_{q}\), the vector field \(\mathsf{U}_{q+1}^{\beta}\) clearly satisfies
  \(
  \mathsf{U}_{q+1}^{\beta}\cdot P
  \equiv
  0
  \)
  and
  \(
  \big(
  \mathsf{D}_{z_{_1}}
  \ad
  \mathsf{U}_{q+1}^{\beta}
  \big)
  =
  \mathsf{U}_{q}^{\beta}
  \),
  which allows to conclude as above, by a shift.
\end{proof}

The vector fields obtained in this way are certainly not enough to correct \eqref{eq:D1VF}, because of the technical limitation to exponents \(\beta\) with \(\abs{\beta}\leq d-q\).
But actually, we will overcome this limitation, 

\begin{lemma}
  \label{lem:moreUb}
  For any \(\beta\in\N^{n}\), their exists a vector field
  with coefficients in \(\C[z_{1},\dotsc,z_{n}]\):
  \[
    \mathsf{U}_{0}^{\beta}
    \bydef
    \sum_{\gamma\leq\beta}
    {\propto}_{\beta,\gamma}(z)\,
    \vf[a]{\gamma}{},
  \] 
  having degree at most \(k+\abs{\beta}-d\), such that:
  \[
    \Lambda_{z_{1}}(
    \mathsf{U}_{0}^{\beta}
    )
    \equiv
    z^{\beta}\,\mathbi{e}_{0}.
  \]
\end{lemma}
\begin{proof}
  In order to construct such vector fields, we will extensively use the following result of Merker, that can be thought of as a much more general analog of the formula
  \(
  \Lambda_{z_{1}}(
  \mathsf{U}_{k+1}^{\beta}
  )
  \equiv
  0
  \):
  \begin{proposition}[Merker~\cite{MR2543663}]
    \label{prop:higher_length}
    For \(\beta\in\N^{n}\) with \(k+1\leq\abs{\beta}\leq d\), fix a multi-index \(\lambda\leq\beta\) with length \(\abs{\lambda}=k+1\). Then, the vector field defined by:
    \[
      \mathsf{T}_{\beta,\lambda}
      \bydef
      \sum_{\gamma\leq\lambda}
      (-1)^{\abs{\gamma}}
      \,
      \frac{\lambda!}{\gamma!(\lambda-\gamma)!}
      \,
      z^{\gamma}
      \,
      \vf[a]{\beta-\gamma}{}.
    \]
    satisfies
    \(
    \mathsf{T}_{\beta,\lambda}
    \cdot
    (\mathsf{D}_{t}^{p}\cdot P)
    \equiv
    0
    \),
    for \(p=0,1,\dotsc,k\),
    identically on \(J_{k}\).
  \end{proposition}
  It is easily reformulated as follows. Take \(\beta,\lambda\in\N^{n}\) with \(\abs{\lambda}=k+1\) such that \(\mathsf{U}_{0}^{\beta-\gamma}\) is defined for any \(\gamma\leq\lambda\), then:
  \[
    \label{eq:merker}
    \mathsf{U}
    \bydef
    \sum_{0<\gamma\leq\lambda}
    (-1)^{\abs{\gamma}}
    \,
    \frac{\lambda!}{(\lambda-\gamma)!\gamma!}
    \,
    z^{\gamma}
    \,
    \mathsf{U}_{0}^{\beta-\gamma}
    \tag{$\ast$}
  \]
  behaves as the non existing vector field ``\(\vf[a]{\beta}{}\)'' up to order \(k\), namely it satisfies:
  \[
    \Lambda_{z_{1}}
    (\mathsf{U})
    \equiv
    z^{\beta}\,
    \mathbi{e}_{0},
  \]
  identically on \(J_{k}\).

  Thus, reasoning by induction on \(\abs{\beta}-d=1,2,\dotsc\), we can construct recursively the vector fields \(\mathsf{U}_{0}^{\beta}\) using \eqref{eq:merker} at each step.
  This construction is not unique, but in all cases the degree is less than \(k+\abs{\beta}-d\), again by induction.
\end{proof}

So, we can extend the result of Lemma \ref{lem:Uq} to all \(\beta\), with no restriction on \(\abs{\beta}-d\).
Then, we deduce as above the following.
\begin{corollary}
  \label{cor:Tj}
  In the slanted directions, for \(j=2,\dotsc,n\) and \(q=0,1,\dotsc,k\),
  the following corrected vector field is tangent to \(J_{k}^{vert}\):
  \[
    \mathsf{T}_{j,q}
    \bydef
    \Big(
    \vf{j}{[q]}
    -
    \hskip-2pt
    \sum_{\abs{\beta}\leq d-1}
    (\beta_{j}+1)\,
    a_{\beta+\mathbi{1}_{j}}\,
    \mathsf{U}_{q}^{\beta}
    \Big).
  \]
\end{corollary}

\subsubsection{Direction of the space of parameters.}
In the direction of the space of parameters, for \(k+1\leq\abs{\beta}\leq d\), any of the vector fields provided by Merker, of the form
\[
  \mathsf{T}_{\beta}
  \bydef
  \vf[a]{\beta}{}
  -
  \sum_{\gamma<\beta}
  {\propto}_{\beta,\gamma}(z)\,
  \vf[a]{\gamma}{},
\] 
is tangent to vertical jets, and for the remaining multi-indices, of length \(\abs{\beta}\leq k\), the following vector fields suits to our purposes:
\[
  \mathsf{T}_{\beta}
  \bydef
  (z_{1}')^{2k-1}
  \Big(
  \vf[a]{\beta}{}
  -
  \sum_{q=0}^{k}
  \big(\mathsf{D}_{z_{1}}^{q}\cdot z^{\beta}\big)\,
  \mathsf{U}_{q}^{\mathbf{0}}
  \Big).
\]
It is obtained by solving the triangular system satisfied by \({\propto}_{0},\dotsc,{\propto}_{k}\) for the vector field
\(
\vf[a]{\beta}{}
-
\sum_{q=0}^{k}
{\propto}_{q}\,
\mathsf{U}_{q}^{\mathbf{0}}
\)
to be a section of \(\mathcal{T}_{k}\), and then by multiplying by the adequate monomial in \(z_{1}'\) to compensate the poles in the fiber.

Since the first family is triangular, and since the corrective parts of the second family are elements of the vector space spanned by \(\vf[a]{\mathbf{0}}{},\vf[a]{\mathbf{1}_{1}}{},\dotsc,\vf[a]{k\,\mathbf{1}_{1}}{}\), the collection of these vector fields span a vector space of codimension \(k+1\) in the direction of the space of parameters.

\section{Low Pole Order Frames on Vertical Jets}

In this section we finish the proof of the main result for the universal hypersurface, and then adapt it to various geometric settings.
\subsubsection{Global generation.}
In order to prove the global generation, in the \(z\)-directions, we use the Corollary \ref{cor:span.z}, the Corollary \ref{cor:Tj} and the simple observation that the vector fields \(\vf[t]{[1]}{}\), \dots,\(\vf[t]{[k]}{}\) are always tangent to the vertical jets ---\,because the variables \(t^{[1]},\dotsc,t^{[k]}\) are not involved in the defining equations\,--- and in the direction of the space of parameters, we use the final observation of section \S\ref{se:construction} above that the vector fields \(\mathsf{T}_{\beta}\) span a vector space of codimension \(k+1\). Because the codimensions agree, this yields indeed that the collection of three families
\[
  \Big\{
    \quad
    \{T_{j,q}\}_{j\geq 2,q=0,\dotsc,k},\ 
    \{\mathsf{T}_{q}\}_{q=0,\dotsc,k},\ 
    \{T_{\beta}\}_{\beta\neq \mathbf{0},\mathbf{1},\dotsc,k\mathbf{1}_{1}}
    \quad
  \Big\},
\]
which contain respectively slanted tangential vector fields, vertical tangential vector fields, and lastly horizontal tangential vector fields,
span the tangent space to vertical jets at points of \(J_{k}^{\mathit{vert}}\) where \(z_{1}'\neq 0\).

\smallskip
Notice that in the argumentation above we have intentionally not mentioned that the vector field \(\vf[a]{\widehat{\alpha}}{}\) does not exist. Actually, there is a good reason for that, since the two conditions \(\abs{\widehat{\alpha}}=d>d-1\) and \(\widehat{\alpha}_{1}=0\) imply that:
\begin{itemize}
  \item the coefficient \(u_{j,\widehat{\alpha}}\) is zero, in other words the non existing vector field \(\vf[a]{\widehat{\alpha}}{}\) is never involved in the construction of \(\mathsf{T}_{j,0}\), for \(q=0\). 
  \item \(\widehat{\alpha}\) is never equal to \(\beta+q\mathbf{1}_{1}\) with \(\abs{\beta}\leq d-1\) and thus it is never involved in the construction of the vector fields \(\mathsf{U}_{q}^{\beta}\), with \(q\geq1\) in Lemma \ref{lem:Uq}, nor in the construction of the vector fields \(\mathsf{U}_{0}^{\beta+q\mathbf{1}_{1}}\) in Lemma \ref{lem:moreUb}.
\end{itemize}

\subsubsection{Pole order of meromorphic prolongations.}
\label{sse:po}
It remains to compute the pole order of the meromorphic prolongations of the corrected vector fields shown above.

Firstly, if \(\abs{\beta}\leq d\), then the pole order of \(\mathsf{U}_{0}^{\beta}\bydef\vf[a]{\beta}{}\) is clearly \(0\) and for \(\abs{\beta}>d\), we have seen that the coefficients of \(\mathsf{U}_{0}^{\beta}\) are polynomials of degree at most \(k+\abs{\beta}-d\) in the variables \(z_{1},\dotsc,z_{n}\),
that can classically be extended as meromorphic function with pole order at most \(k+\abs{\beta}-d\) on \(\P^{n+1}\);
hence by additivity with respect to product and subadditivity with respect to sum, using the definition in Lemma \ref{lem:Uq}, the pole order of \(\mathsf{U}_{q}^{\beta}\) is:
\begin{equation}
  \label{eq:poUq}
  \mathrm{p.o}(\mathsf{U}_{q}^{\beta})
  =
  \begin{cases}
    q
    &\text{if }\abs{\beta}+q\leq d 
    \\
    q+k+\abs{\beta}+q-d
    &\text{if }\abs{\beta}+q>d.
  \end{cases}
\end{equation}

\begin{subequations}
  Then, by considering the successive derivations of \(z_{i}=Z_{i}/Z_{0}\), it is easy to see that the pole order along the hyperplane at infinity \((Z_{0}=0)\)  of the meromorphic prolongation to \(\P^{n+1}\) of \(z_{i}^{(p)}\) is \(p+1\).
  This yields that for \(\mu\in\N^{k}\), the pole order of the meromorphic continuation to \(\P^{n+1}\) of the monomial
  \[
    z_{i}^{(\mu)}
    \bydef
    {z_{i}^{(1)}}^{\mu_{1}}
    \dotsm
    {z_{i}^{(k)}}^{\mu_{k}}
    \qquad
    {\scriptstyle(i=1,\dotsc,n)}
  \] 
  is, by additivity:
  \begin{equation}
    \label{eq:po.1}
    \mathrm{p.o}\big(z_{i}^{(\mu)}\big)
    =
    \sum_{p}(p+1)\mu_{p}
    =
    \Abs{\mu}+\abs{\mu}.
  \end{equation}
  Consequently, by subadditivity, the meromorphic continuation to \(\P^{n+1}\) of the Bell polynomial
  \(
  \mathbi{B}_{p,q}(z_{1})
  =
  \sum_{\Abs{\mu}=q,\abs{\mu}=p}
  \frac{\abs{\mu}!}{\mu!}\,
  z_{1}^{(\mu)}
  \) 
  has pole order:
  \begin{equation}
    \label{eq:po.2}
    \mathrm{p.o}\left(\mathbi{B}_{p,q}\right)
    =
    p+q.
  \end{equation}
\end{subequations}

Using \eqref{eq:po.1} and \eqref{eq:po.2}, it is elementary to compute the pole orders of the constructed vector fields by coming back to their expressions in the standard jet coordinate system.

\begin{subequations}
  The coefficients of \(\vf{j}{[q]}\) are the Bell polynomials \(\mathbi{B}_{q,q}(z_{1}),\dotsc,\mathbi{B}_{q,k}(z_{1})\), whence: 
  \[ \mathrm{p.o.} \big( \vf{j}{[q]} \big) = q+k.\]
  This and \eqref{eq:poUq} yields:
  \begin{equation}
    \label{eq:poTjq}
    \mathrm{p.o}(\mathsf{T}_{j,q})
    =
    \begin{cases}
      k+q&\text{if }q\leq 1\\
      k-1+2q&\text{if }q\geq 2.
    \end{cases}
  \end{equation}

  The coefficients of the vector fields
  \[
    \mathsf{T}_{q}
    \bydef
    \sum_{m=1}^{k}
    \mathbi{B}_{m,\ell}[t]\,
    \vf[t]{}{[m]}
    =
    -
    \sum_{i=1,\dotsc,n}\;
    \sum_{p=1}^{k-\ell+1}
    p\,z_{i}^{(p)}\,\vf{i}{(p+\ell-1)}
  \]
  are the monomials \(z_{i}^{(1)},\dotsc,z_{i}^{(k-q)}\) whence:
  \begin{equation}
    \label{eq:poTsym}
    \mathrm{p.o}(\mathsf{T}_{q})
    =
    2\,(k-q).
  \end{equation}

  It remains to compute the pole order of the meromorphic prolongation of the vector fields \(\mathsf{T}_{\beta}\) and we will shortly show the following.
  \begin{equation}
    \label{eq:poTb}
    \mathrm{p.o}(\mathsf{T}_{\beta})
    =
    \begin{cases}
      k+1&\text{if }k+1\leq\abs{\beta}\leq d\\
      4k+\abs{\beta}-2&\text{if }\abs{\beta}\leq k\\
    \end{cases}
  \end{equation}
\end{subequations}

These observations \eqref{eq:poTjq}\,--\,\eqref{eq:poTb} yield the constants \((5\,k-2)\) in the main theorem and the other constant is \(1\) because everything is clearly linear in \(a_{\alpha}\).

Notice that this pole order is \(1\) more than the pole orders computed in small dimensions by P\u{a}un (\(7\) for \(k=2\)) and Rousseau (\(12\) for \(k=3\)), but this is a fair price to pay for obtaining the global generation on the subset of invertible jets, and not anymore in the complement of the zero locus of a certain determinant depending on the equation of the hypersurface. This is an important detail in order to get results towards the \emph{strong} Green-Griffiths conjecture (\textit{cf.} \cite{MR2593279,arXiv:1402.1396}).

\begin{proof}
  [Proof of \eqref{eq:poTb}]
A quick induction based on the chain rule \(\mathsf{D}_{z_{1}}\bydef\mathsf{D}_{t}/z_{1}'\) shows that 
for \(q\leq k\) and \(\beta\in\N^{n}\), there exist combinatorial coefficients \({\propto}_{\bullet}\) such that:
\[
  \mathsf{D}_{z_{1}}^{q}\cdot z^{\beta}
  =
  \sum_{\substack{
    \abs{\lambda}+\abs{\mu^{1}}+\dotsb+\abs{\mu^{n}}=\abs{\beta}+q-1\\
    \Abs{\mu^{1}}+\dotsb+\Abs{\mu^{n}}=2q-1\\
  }}
  {\propto}_{\lambda,\mu^{1},\dotsc,\mu^{n}}(\beta,q)\;
  z^{\lambda}\;
  \frac{
    \Big(
    z_{1}^{(\mu^{1})}
    \dotsm
    z_{n}^{(\mu^{n})}
  \Big)}
  {{z_{1}'}^{2q-1}}.
\]
So after multiplying by \({z_{1}'}^{2q-1}\) one obtains a polynomial with pole order:
\[
  \mathrm{p.o.}
  \big(
  {z_{1}'}^{2q-1}\,
  \mathsf{D}_{z_{1}}^{q}\cdot z^{\beta}
  \big)
  =
  \abs{\lambda}
  +
  (
  \abs{\mu^{1}}
  +
  \Abs{\mu^{1}}
  )
  +
  \dotsb
  +
  (
  \abs{\mu^{n}}
  +
  \Abs{\mu^{n}}
  )
  =
  \abs{\beta}+3q-2.
\]
Thus, using subadditivity
\begin{align*}
  \mathrm{p.o}\;
  \big(
  (z_{1}')^{2k-1}
  \big(\mathsf{D}_{z_{1}}^{q}\cdot z^{\beta}\big)\,
  \mathsf{U}_{q}^{\mathbf{0}}
  \big)
  &\leq
  \mathrm{p.o}\big((z_{1}')^{2(k-q)}\big)
  +
  \mathrm{p.o.}
  \big(
  {z_{1}'}^{2q-1}\,
  \mathsf{D}_{z_{1}}^{q}\cdot z^{\beta}
  \big)
  +
  \mathrm{p.o.}
  \big(
  \mathsf{U}_{q}^{\mathbf{0}}
  \big)
  \\
  &\leq
  4(k-q)
  +
  \abs{\beta}+3q-2
  +
  q
  =
  4k+\abs{\beta}-2,
\end{align*}
which indeed leads to \(\mathrm{p.o}(\mathsf{T}_{\beta})=4k+\abs{\beta}-2\) for \(\abs{\beta}\leq k\).
\end{proof}

This ends the proof of the main result in the case of the universal hypersurface.

\subsubsection{Global generation for complete intersections.}
For complete intersections, the strategy described in section \S\ref{se:construction} directly applies, since the only common variables between two of the \(c\) systems of algebro-differential equations corresponding to each of the \(c\) defining polynomials are the \(z\)-variables.
Fixing the vertical part \(\mathsf{V}\) (possibly \(=0\)), it is possible to solve the problem for each system separately, then adding all the corrective parts, one obtains a corrective part for all the \(c\) systems together.

\subsubsection{Modifications needed for the logarithmic case.}
In order to treat the logarithmic case in a very similar way as the compact case, 
we first straighten out the universal hypersurface, following the strategy of Rousseau in \cite{MR2552951}. As in the compact case, we start by the case \(c=1\).

In order to straighten out the universal family \(\mathcal{H}_{d}\), given in the system of coordinates \(\bigl([Z],[A]\bigr)\) on \(\P^{n}\times \P^{n_{d}}\) by
\(
\mathcal{H}_{d}
\bydef
\{
  0
  =
  \sum_{\abs{\alpha}=d}
  A_{\alpha}\,Z^{\alpha}
\}
\subset
\P^{n}\times
\P^{n_{d}}
\),
introduce a new homogeneous "\(Z\)-coordinate" \(W\in \C\) associated with a new homogeneous "\(A\)-coordinate" \(A_{0}\in\C\), thought of as the coefficient of the monomial \(W^{d}\). 
Accordingly consider the zero set:
\[
  \mathcal{X}
  \bydef
  \biggl\{
    A_{0}\,
    W^{d}
    =
    \sum_{\abs{\alpha}=d}
    A_{\alpha}\,
    Z^{\alpha}
  \biggr\}
  \subset
  \P^{n+1}\times\P^{n_{d}+1}.
\]

There is a natural \textsl{forgetful map}:
\[
  \pi
  \colon
  \P^{n+1}\times\P^{n_{d}+1}
  \setminus
  \left(\{\forall Z_{i}=0\}\cup\{\forall A_{\alpha}=0\}\right)
  \to
  \P^{n}\times\P^{n_{d}},
\]
that consists in erasing both \(W\) and \(A_{0}\). 
Notice that:
\[
  \mathcal{X}
  \cap
  \left(\{\forall Z_{i}=0\}\cup\{\forall A_{\alpha}=0\}\right)
  \subset
  \mathcal{X}
  \cap
  \left(\{A_{0}=0\}\cup\{\forall A_{\alpha}=0\}\right)
\]
Indeed, if \(Z=0\), the equation of \(\mathcal{X}\) becomes:
\(
A_{0}\,W^{d}
=
0.
\)
This implies that either \(A_{0}\) or \(W\) must be zero. 
But \(W\) cannot be zero, because the homogeneous coordinates of the point \([W:0:\dotso:0]\) in the projective space \(\P^{n+1}\) cannot be all simultaneously zero. Thus \(A_{0}\) must be zero.

Let \(\mathcal{X}^{*}\) be the restriction of \(\mathcal{X}\) to the affine set, pointed at the origin:
\[
  \{A_{0}\neq 0\}\setminus\{[1:0:\dotso:0]\}
  \simeq
  \C^{n+1}\setminus\{0\}.
\]
Then, the projection \(\pi\vert_{\mathcal{X}^{*}}\colon\mathcal{X}^{*}\to\P^{n}\times \P^{n_{d}}\) is well defined and moreover, it is a branched cover of degree \(d\) that ramifies exactly over \(\mathcal{H}_{d}\).
Since \(A_{0}\neq 0\), the inverse image of the universal family \(\mathcal{H}_{d}\) under this projection identifies with the (straight) hyperplane 
\[
  D
  \bydef
  (\pi\vert_{\mathcal{X}^{*}})^{\moinsun}(\mathcal{H}_{d})
  =
  \{W=0\}.
\]

The map \(\pi\colon(\mathcal{X}^{*},D)\to(\P^{n}\times \P^{n_{d}},\mathcal{H}_{d})\) is therefore a log-morphism (\cite{MR637060}),
that induces a canonical holomorphic map on the spaces of jets of logarithmic curves:
\[
  \pi_{[k]}
  \colon
  J_{k}\mathcal{X}^{*}(-\log D)
  \to
  \pi^{\star}J_{k}(\P^{n}\times \P^{n_{d}})(-\log\mathcal{H}),
\]
that is clearly dominant, as \(d\pi_{[k]}\) is of maximal rank.
This projection \(\pi_{[k]}\) also send vertical jets on vertical jets.
We will thus study the vertical logarithmic jets upstairs, where it is easier to use logarithmic jet-coordinates.

\smallskip

As long as the \(w\)-component of the vector field is zero, nothing change compared to \S\ref{se:construction}, except that there is no multi-index \(\widehat{\alpha}\) to remove, and we have all the tangential vector fields of \S\ref{se:construction}.

For the supplementary logarithmic direction \(w\), by the Faà di Bruno formula, one has:
\[
  w^{[q]}
  =
  \exp(\log w)\,
  \sum_{\Abs{\mu}=q}
  \frac{{(\log w)^{[1]}}^{\mu_{1}}}{\mu_{1}!}
  \dotsm
  \frac{{(\log w)^{[k]}}^{\mu_{k}}}{\mu_{k}!}
  \quad
  {\scriptstyle(q=0,1,\dotsc,k)},
\]
which implies the dual relations:
\[
  \vf[(\log w)]{}{[p]}
  =
  \sum_{q=p}^{k}
  w^{[q-p]}\,
  \vf[w]{}{[q]}
  \quad
  {\scriptstyle(q=0,1,\dotsc,k)}.
\]
In particular, notice for later use that for \(P\in\C[w,z_{1},\dotsc,z_{n}]\):
\[
  \vf[\log w]{}{}\cdot P
  =
  w\,
  \vf[w]{}{}\cdot P
  \qquad
  \text{and}
  \qquad
  \vf[(\log w)]{}{[p]}\cdot P
  =
  0
  \quad
  {\scriptstyle(p=1,\dotsc,k)}.
\]
\begin{lemma}
  For \(p=0,1,\dotsc,j\), the vector field \(\mathsf{V}=\vf[(\log w)]{}{[p]}\) does not satisfy \(\Lambda_{z_{1}}(\mathsf{V})=0\), because it produces a nonzero entry in the \((p+1)\)-th line:
  \[
    \Lambda_{z_{1}}\big(
    \vf[(\log w)]{}{[p]}
    \big)
    =
    (-1)^{p}\,p!\;
    d\,w^{d}\;
    \mathbi{e}_{p}.
  \]
\end{lemma}
\begin{proof}
  For \(p\geq1\), applying the Leibniz rule, as follows,
  \begin{align*}
    \big(
    \mathsf{D}_{z_{1}}
    \ad
    \vf[(\log w)]{}{[p]}
    \big)
    &=
    \sum_{q=p}^{k}
    \big(\mathsf{D}_{z_{1}}\cdot w^{[q-p]}\big)\,
    \vf[w]{}{[q]}
    +
    \sum_{q=p}^{k}
    w^{[q-p]}\,
    \big(
    \mathsf{D}_{z_{1}}
    \ad
    \vf[w]{}{[q]}
    \big)
    \\
    &=
    \sum_{q=p}^{k}
    (q-p+1)\,w^{[q-p+1]}\,
    \vf[w]{}{[q]}
    +
    \sum_{q=p}^{k}
    w^{[q-p]}\,
    \big(
    -q\,
    \vf[w]{}{[q-1]}
    \big),
  \end{align*}
  and then, shifting the indices in the second sum, and adding zero in the first sum by considering also the term for \(q=p-1\):
  \begin{align*}
    \big(
    \mathsf{D}_{z_{1}}
    \ad
    \vf[(\log w)]{}{[p]}
    \big)
    &=
    \sum_{q=p-1}^{k}
    (q-p+1)\,w^{[q-p+1]}\,
    \vf[w]{}{[q]}
    +
    \sum_{q=p-1}^{k-1}
    w^{[q-p+1]}\,
    \big(
    -(q+1)\,
    \vf[w]{}{[q]}
    \big)
    \\
    &=
    -p
    \sum_{q=p-1}^{k}
    w^{[q-(p-1)]}\,
    \vf[w]{}{[q]}
    +
    (k+1)\,
    w^{[k+1-p]}\,
    \vf[w]{}{[k]},
  \end{align*}
  one obtains that for \(p\geq1\):
  \[
    \big(
    \mathsf{D}_{z_{1}}
    \ad
    \vf[(\log w)]{}{[p]}
    \big)
    =
    -p
    \vf[(\log w)]{}{[p-1]}
    +
    (k-p+1)\,w^{(k-p+1)}\,\vf[w]{}{[k]}.
  \]
  Similarly for \(p=0\) one obtains:
  \[
    \big(
    \mathsf{D}_{z_{1}}
    \ad
    \vf[\log w]{}{}
    \big)
    =
    0.
  \]

  By induction, this yields that, for \(q\leq p\):
  \[
    \big(
    \mathsf{D}_{z_{1}}^{q}
    \ad
    \vf[(\log w)]{}{[p]}
    \big)
    =
    (-1)^{p-q}\,\frac{p!}{(p-q)!}\,
    \vf[(\log w)]{}{[p-q]}
    +
    \sum_{r=k+1-q}^{k}
    {\propto}_{q,p}(w^{(k+1-p)},\dotsb,w^{(k+q-p)})\,
    \vf[w]{}{[q]},
  \]
  for some polynomial coefficients \({\propto}_{q,p}\),
  and for \(q\geq p\)
  \[
    \big(
    \mathsf{D}_{z_{1}}^{q}
    \ad
    \vf[(\log w)]{}{[p]}
    \big)
    =
    \sum_{r=k+1-q}^{k}
    {\propto}_{q,p}(w^{(k+1-p)},\dotsb,w^{(k+q-p)})\,
    \vf[w]{}{[q]},
  \]
  for some polynomial coefficients \({\propto}_{q,p}\).
  Now, recall that for \(P\in\C[w,z_{1},\dotsc,z_{n}]\)
  one has
  \(
  \vf[\log w]{}{}\cdot P
  =
  w\,
  \vf[w]{}{}\cdot P
  \) and \(
  \vf[(\log w)]{}{[p]}\cdot P
  =
  0
  \quad
  {\scriptstyle(p=1,\dotsc,k)}
  \),
  and notice that of course
  \(
  \vf[w]{}{[p]}\cdot P
  =
  0
  \quad
  {\scriptstyle(p=1,\dotsc,k)}
  \),
  in order to conclude, as announced that for \(q\neq p\):
  \[
    \big(
    \mathsf{D}_{z_{1}}^{q}
    \ad
    \vf[(\log w)]{}{[p]}
    \big)
    \cdot
    \big(
    w^{d}-\sum_{\abs{\alpha}\leq d} a_{\alpha}\,z^{\alpha}
    \big)
    =
    0,
  \]
  and for \(q=p\):
  \[
    \big(
    \mathsf{D}_{z_{1}}^{p}
    \ad
    \vf[(\log w)]{}{[p]}
    \big)
    \cdot
    \big(
    w^{d}-\sum_{\abs{\alpha}\leq d} a_{\alpha}\,z^{\alpha}
    \big)
    =
    (-1)^{p}\,p!\;
    \vf[\log w]{}{}\cdot w^{d}
    =
    (-1)^{p}\,p!\;
    d\,w^{d},
  \]
  which concludes the proof.
\end{proof}

\begin{corollary}
  \label{cor:Tw}
  In the logarithmic direction, for \(q=0,1,\dotsc,k\),
  the following corrected vector field is tangent to \(J_{k}^{vert}\):
  \[
    \mathsf{T}_{w,q}
    \bydef
    \vf[w]{}{[q]}
    +
    (-1)^{q}\,q!\,d\;
    \Big(
    \hskip-2pt
    \sum_{\abs{\alpha}\leq d}
    a_{\alpha}\,
    \mathsf{U}_{q}^{\alpha}
    \Big).
  \]
\end{corollary}
\begin{proof}
  Indeed:
  \[
    w^{d}=\sum_{\abs{\alpha}\leq d}a_{\alpha}\,z^{\alpha}
    \Longrightarrow
    \big(
    \mathsf{D}_{z_{1}}^{q}
    \ad
    \mathsf{T}_{w,p}
    \big)
    \Big(w^{d}-\sum_{\abs{\alpha}\leq d} a_{\alpha}\,z^{\alpha}\Big)
    =
    0.
  \]
  Thus \(\mathsf{T}_{w,q}\) is a section of \(\mathcal{T}_{k}\), and in particular it is tangent to \(J_{k}^{\mathit{vert}}\).
\end{proof}

Notice that here the polynomial is of degree \(d\) and not \(d-1\) as in the compact case, but this is not a problem, 
since in the logarithmic case, we have not removed a multi-index \(\widehat{\alpha}\).

\subsubsection{Modifications needed for the logarithmic case with several components.}
In order to straighten out \(\mathcal{H}_{d_{1},\dotsc,d_{c}}\), introduce \(c\) auxiliary variables \(W_{1}\), \dots, \(W_{c}\) together with the associated parameters \(A_{\mathbf{0}}^{1}\), \dots, \(A_{\mathbf{0}}^{c}\), and consider the complete intersection
\[
  \mathcal{Y}
  \subset
  \{Z_{0}\neq0\}
  \times
  \{A_{\mathbf{0}}^{1}\neq0\}
  \times
  \dotsb
  \times
  \{A_{\mathbf{0}}^{c}\neq0\},
\]
defined by the equations:
\[
  {w_{1}}^{d}=\sum_{\abs{\alpha}\leq d_{1}} a_{\alpha}^{1}\,z^{\alpha}\quad,\qquad
  \dotsc\quad,\qquad
  {w_{c}}^{d}=\sum_{\abs{\alpha}\leq d_{c}} a_{\alpha}^{c}\,z^{\alpha}.
\]
These \(c\) equations have only the \(z\)-variables in common. Since we are able to solve the problem for a fixed ``\(z\)-part'' (possibly equal to zero) with \(c=1\), it is possible to solve the problem for each equation separately. Then adding all the corrective parts we get a corrective part for the system of \(c\) equations, since the ``\(w\)-parts'' and ``\(a\)-parts'' interact with only one of the \(c\) equations, and vanish on the others.

\subsection*{\bfseries Acknowledgments}
I want to thank \textsl{Jérémy Guéré} for suggestions concerning presentation and \textsl{Christophe Mourougane} for interesting discussions on geometric jet coordinates. I warmly thank \textsl{Jean-Pierre Demailly} for friendly explanations in his office, which influenced the definition of the geometric jet coordinates. Lastly, I would like to gratefully thank my thesis advisor \textsl{Joël Merker} for his support, his very careful reading and the proposal of relevant lines of thinking.

\bibliographystyle{amsplain}
\bibliography{$LATEX/these}
\vfill
\end{document}